\documentclass[11pt,a4paper]{amsart}
\usepackage[backend=biber,
style=alphabetic,
url=true,
nohashothers=false,
maxnames=50,
maxalphanames=5]{biblatex}
\usepackage[utf8]{inputenc}
\usepackage[T1]{fontenc}
\usepackage{lmodern}
\usepackage{multirow}
\usepackage[section]{placeins}

\usepackage{amsmath}
\usepackage{amsfonts}
\usepackage{amssymb}
\usepackage{amsthm}
\usepackage{mathrsfs}
\usepackage{footmisc}
\usepackage{empheq}
\usepackage{stmaryrd}

\usepackage{makecell}
\usepackage{MnSymbol}
\usepackage{float}
\restylefloat{table}
\usepackage[usenames, dvipsnames]{color}
\usepackage{tikz}
\usepackage{tikz-cd}
\usepackage{caption,subcaption}
\usepackage{wrapfig}
\usepackage{float} 
\usepackage{multicol}
\usepackage{empheq}
\usepackage{array}
\usepackage[thinlines]{easytable}
\usepackage{pgfplots}
\usepackage{enumitem}
\usepackage{calc}
\pgfplotsset{compat=1.13} 

\usepackage{todonotes}

\newlength{\mylongest}
\SetLabelAlign{CenterWithParen}{\makebox[\mylongest]{(#1)}}
\usepackage{colortbl}

\SetLabelAlign{CenterWithParen}{\makebox[2mm]{(#1)}}

\newcommand\numberthis{\addtocounter{equation}{1}\tag{\theequation}}

\calclayout

\definecolor{liens}{rgb}{1,0,0}
\usepackage[colorlinks=true, linkcolor=blue,
hyperfootnotes=blue,citecolor=blue,urlcolor=blue]{hyperref}

\usepackage{thmtools}

\theoremstyle{plain}
\newtheorem{theorem}{Theorem}[section]
\newtheorem*{theorem*}{Theorem}
\newtheorem{lemma}[theorem]{Lemma}
\newtheorem{corollary}[theorem]{Corollary}
\newtheorem{proposition}[theorem]{Proposition}

\renewcommand{\thmcontinues}[1]{continued}

\theoremstyle{definition}
\newtheorem{definition}[theorem]{Definition}
\newtheorem{defprop}[theorem]{Definition/Proposition}

\newtheorem{example}[theorem]{Example}

\theoremstyle{remark}
\newtheorem{remark}[theorem]{Remark} 

\numberwithin{equation}{section}
\numberwithin{figure}{section}
\numberwithin{table}{section}
\theoremstyle{theorem}

\def\C{\mathbb{C}}

\def\id{\text{id}}

\def\eqdef{\overset{\mathrm{def}}{=}}

\def\st{\, : \,}          
\def\id{\mathrm{id}}

\def\Z{\mathbb{Z}}
\def\C{\mathbb{C}}
\def\R{\mathbb{R}}
\def\Q{\mathbb{Q}}
\def\N{{\mathbb N}}
\def\K{{\mathbb K}}

\def\P1{\mathbb{P}^{1}}

\def\beq{\begin{equation}}
  \def\eeq{\end{equation}}
\def\Et{E_t}
\def\Etproj{\overline{E_t}}

\def\P2{\mathbb{P}^{2}}

\def\K{\mathbb{K}}

\def\ord{\mbox{ord }}

\def\P1{\mathbb{P}^{1}}

\def\calC{{\mathcal{C}}}
\def\calS{\mathcal{S}}

\def\calA{{\mathcal{A}}}
\def\calB{{\mathcal{B}}}

\def\ord{{\rm ord}}

\def\cL{\mathcal{L}}

\def\P{\mathbb{P}}

\def\Ftld{\widetilde{F}}
\def\Gtld{\widetilde{G}}
\def\ftld{\widetilde{f}}
\def\gtld{\widetilde{g}}

\def\gamtld{\widetilde{\gamma}}
\DeclareFontFamily{U}{mathx}{}
\DeclareFontShape{U}{mathx}{m}{n}{<-> mathx10}{}
\DeclareSymbolFont{mathx}{U}{mathx}{m}{n}
\DeclareMathAccent{\widehat}{0}{mathx}{"70}
\DeclareMathAccent{\widecheck}{0}{mathx}{"71}

\def\llpar{(\!(}
\def\rrpar{)\!)}

\def\Fcap{\widecheck{F}}
\def\Gcap{\widecheck{G}}

\def\Fpar{\mathbb{F}}

\newcommand\xqed[1]{%
  \leavevmode\unskip\penalty9999 \hbox{}\nobreak\hfill
  \quad\hbox{#1}}
\newcommand\exqed{\xqed{\tiny $\blacksquare$}}

\title[Walks in the quadrant with interacting boundaries: genus zero case]{Walks in the quadrant with \\ interacting boundaries: genus zero case}
\author{Pierre Bonnet}

\begin{document}
\maketitle

\begin{abstract}
  The study of \emph{lattice walks} restricted to the first quadrant
  has shed a lot of interest in the past twenty years. In particular, there
  has been an important effort to classify models of \emph{weighted
    walks} with small steps with respect to the algebraic-differential nature of their
  generating function.  The techniques that were developed in the
  course of this work are now applied to different extensions of those
  walks. One of these extensions, called \emph{walks with interacting
    boundaries}, consists in accounting for the number of contacts of the walk with
  the axes, with motivation coming from statistical physics.  These
  contacts are encoded as two additional parameters for the generating
  function, the \emph{Boltzmann weights}.

  For one notable family of models, called \emph{genus zero models},
  we establish in this paper the complete classification of their
  generating function, for all real values of the parameters. We do this
  by adapting to this more general case a method due to
  Dreyfus, Hardouin, Roques and Singer, used
  in the former classification, and which consists
  in studying the rational solutions to a \emph{$q$-difference equation}.
  In almost all cases, we show that the generating function is
  hypertranscendental, regardless of the values of the weights. In the
  remaining cases, we prove that specific algebraic relations between
  the Boltzmann weights make the generating function $\N$-algebraic or
  $\N$-rational, contrasting with the interaction-less case.
\end{abstract}

\section*{Introduction}

Lattice walks restricted to a cone are ubiquitous.  Aside from their
own interest (they naturally model random processes), these objects
are general enough to represent many classes of discrete objects,
including trees, permutations, planar maps, queuing processes\ldots\,

\begin{figure}[h]
  \centering
  \includegraphics{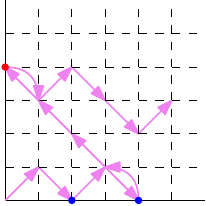}
  \captionsetup{labelformat=empty}
  \caption{A walk in the quadrant}
\end{figure}

One important type of lattice walks are the two dimensional walks
restricted to the first quadrant. In this setting, a \emph{set of steps}
$\calS$ is set, and the goal is to enumerate the walks that use the steps
in $\calS$, and whose coordinates $(i,j)$ are nonnegative at all
times. The combination of two constraints ($i\ge 0$ and
$j\ge 0$) make the nature of the problem extremely dependent on the
set of steps $\calS$, and highly nontrivial.

As a support for this
claim, we mention the remarkable work of the classification
of two dimensional quadrant walks with small steps (i.e.
with $\calS \subset \{-1,0,1\}^2$), whose goal was to
determine the algebraic-differential nature of the generating
function $Q(x,y,t)$
of such quadrant walks, counted with respect to ending coordinates and length. This systematic exploration
showed a wide range of different behaviours depending on the set
of steps. This achievement involved many results from
enumerative combinatorics~\cite{BMM,BMGessel}, probability theory~\cite{DenisovWachtel,FIM}, computer algebra~\cite{BBMR16}
and differential Galois theory~\cite{DHRS,DreyfusHardouinRoquesSingerGenuszero2}, whose collective contributions
gave rise to many effective tools to study
\emph{catalytic variables} equations. A reference
sketching the combination of these different results
to get the classification can be found for instance
in~\cite{dreyfusEnumerationWeightedQuadrant2024}.
The systematic study of quadrant walks continues to this day with
other types of walks, for instance walks with arbitrarily large steps
\cite{bostan2018counting,BHorb}, three-quadrant walks
\cite{bousquet-melouEnumerationThreequadrantWalks2023}, or the focus
of the current paper, \emph{walks with interacting boundaries}.

This extension, first introduced in \cite{tabbaraExactSolutionTwo2014},
consists in accounting for the number of contacts
of the walk with the axes $i=0$ and $j=0$ (the \emph{interaction}).
The study of such problems leads to the study of the
\emph{phase transitions} of the
model~\cite{rensburgStatisticalMechanicsInteracting2015}. More precisely,
one parametrizes the tendency of the model to stick to the boundary
$i=0$ (resp. $j=0$) via the \emph{Boltzmann weight} $b \in \R^+$ (resp. $a \in \R^+$).
The qualitative behaviour of the system changes depending on the Boltzmann
weights, thus defining its phases. Counting the walks relative to the
interaction statistics is done through functional equations that generalize
those that appeared earlier, when ignoring this statistics.
As a result, some of the techniques developed for counting walks
in the quadrant may be applied to this setting
(see~\cite{dreyfus2019differential}).
For small steps models, they
allowed authors to classify the generating
function of some models for some Boltzmann
weights~\cite{beaton_quarter-plane_2021,beaton_exact_2019}, or
even to solve them to get the full phase transition diagram
such as in~\cite{tabbaraExactSolutionTwo2014}.

\subsection*{Contribution}
In this paper, we focus on the so-called \emph{weighted models of genus zero},
which correspond to the five sets of steps below.
\begin{figure*}[h!]
  \centering
  \includegraphics[width=.7\textwidth]{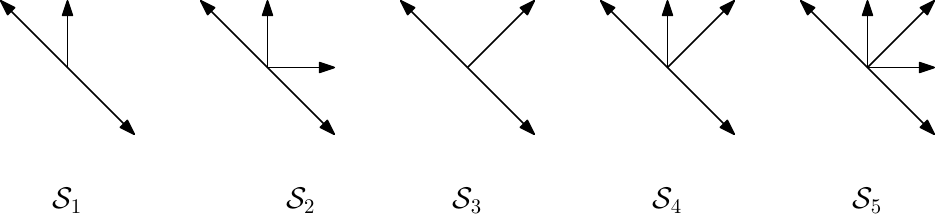}
\caption*{The five considered sets of steps}
\end{figure*}
For these sets of steps, we give the full classification of the
algebraic-differential nature of the generating function $Q(x,y)$ of weighted walks
with interacting boundaries with respect to the
variables $x$ and $y$. Recall that for every
weighting $(d_{v})_{v\in\calS}$ attached to the steps
and Boltzmann weights $a$ and $b$, the generating
function $Q(x,y)$ is defined as
\[
  Q(x,y) = \sum_{\mbox{$w$ walk}} \mbox{weight}(w) x^i y^j t^n,
\]
where for $w$ a walk, $(i,j)$ are its ending coordinates, $n$ is its length,
and $\mbox{weight}(w)$ is a monomial
depending on the weights $(d_v)_{v \in \calS}$, $a$ and $b$ (defined
in Section~\ref{sec:weighted-walks}).

To do this, we adapt the strategy of~\cite{DreyfusHardouinRoquesSingerGenuszero}
that was used to classify the generating function of weighted walks
based upon on the sets of steps of Figure~\ref{fig:2Dwalk}.
This amounts to study the rational solutions of
some functional equation (a  \emph{$q$-difference equation}),
whose coefficients depend on the parameters $a$, $b$ and $(d_v)_{v \in \calS}$.
This $q$-difference equation for the case of walks with
interacting boundaries will be obtained
in Section~\ref{sec:qdiff-equ}.

In Section~\ref{sec:class-strat} we exploit the symmetries of the
$q$-difference equation to reduce the classification of its solutions to
the study of two decoupling problems.
We develop in Section~\ref{sect:decoupl_eq}
criteria based on pole propagation
to test the existence of rational solutions to
decoupling equations of a special form.
These criteria reduce the classification of $Q(x,y)$
to relations between the weights,
and they provide a concise way to explore
the space of parameters, to either prove the nonexistence of such
decouplings or to find solutions. These methods call
for eventual generalizations (Section~\ref{sec:other-q-difference}).

Using this technique, we ultimately obtain the form of the following theorem:
\begin{theorem*}[Theorem~\ref{thm:thm_clas}, Section~\ref{sec:full-classification}]
  For any weighted genus~$0$ model, the generating function $Q(x,y)$ of
  weighted walks in the quadrant with interacting boundaries has the following
  nature in the variables $x$ and $y$:
  \begin{enumerate}
  \item For the sets of steps $\calS_1$ or $\calS_2$ and Boltzmann
    weights satisfying $a+b=ab$, the generating function $Q(x,y)$ is
    \textbf{rational} with specializations $Q(x,0)$ and $Q(0,y)$
    respectively equal to
    \begin{align*} Q(x,0) &=\frac{1}{1 - x \displaystyle \frac{a
                            d_{1,0} t + ab d_{1,-1} d_{0,1} t^2}{1 - ab d_{1,-1} d_{-1,1} t^2}}, &
                                                                                                   Q(0,y) &= \frac{1}{1 - y \displaystyle \frac{b d_{0,1} t + ab d_{-1,1}
                                                                                                            d_{1,0} t^2}{1 - ab d_{1,-1} d_{-1,1} t^2}}.
    \end{align*}
  \item For the set of steps $\calS_3$ and Boltzmann weights $a=b=2$,
    the generating function $Q(x,y)$ is \textbf{algebraic} of degree at
    most $4$, with specializations $Q(x,0)$ and $Q(0,y)$ respectively equal
    to
    \begin{align*} Q(x,0) &= \frac{1}{\sqrt{1 - \displaystyle x^2
                            \frac{4 d_{1,1} d_{1,-1} t^2}{1-4 d_{1,-1} d_{-1,1} t^2}}}, & Q(0,y)
      &= \frac{1}{\sqrt{1 - \displaystyle y^2 \frac{4 d_{1,1} d_{-1,1}
        t^2}{1-4 d_{1,-1} d_{-1,1} t^2}}}.
    \end{align*}
  \item In every other case, the series $Q(x,y)$ is \textbf{non
      $\boldsymbol{x}$-D-algebraic nor $\boldsymbol{y}$-D-algebraic} (meaning $Q(x,y)$
    satisfies no polynomial differential equation in $x$ nor in $y$ for any
    choice of $x$, $y$, $t$ and weighting $(d_v)_{v \in \calS}$, $a$ and $b$).
  \end{enumerate}
\end{theorem*}

In~\cite{DreyfusHardouinRoquesSingerGenuszero}
where the Boltzmann weights $a$ and $b$ are both equal to one,
the generating functions of the models were found to
be all non $x$-D-algebraic nor $y$-D-algebraic. The addition
of the Boltzmann weights $a$ and $b$ allows us to find algebraic models.

\subsection*{Organization of the paper}

In Section~\ref{sec:study-walks-thro}, we recall standard definitions and
facts in the study of quadrant walks, mainly their statistics, the
weighting associated to a walk given a weighting $((d_v)_{v\in\calS},a,b)$, the generating function $Q(x,y)$ of
such walks, and the algebraic-differential classification of bivariate
power series. We then focus on the five sets of steps of
Figure~\ref{fig:2Dwalk}, for which the kernel curve
has genus~$0$.
As a result, the kernel curve admits a rational parametrization $(x(s),y(s))$,
for which we will recall basic facts, among which
the existence of an automorphism $\sigma(s) \eqdef  q s$ for some
real number $q$ which is not a root of unit.
We then proceed to evaluate the functional equation
for $Q(x,y)$ on this curve, this way obtaining two independent functional
equations ($q$-difference equations) on the functions
$\Ftld(s)=Q(x(s),0)$ and $\Gtld(s)=Q(0,y(s))$.
We then compare the algebraic-differential properties of these two functions
with those of $Q(x,y)$, thus reducing to the study of $\Ftld(s)$ and $\Gtld(s)$
through these $q$-difference equations.

In Section~\ref{sec:class-strat}, we thus devise the strategy for determining
the algebraic-differential nature of $\Ftld(s)$ and $\Gtld(s)$. The analytic
properties of $q$-difference equations being rigid enough, a theorem due to
Ishizaki allows us to reduce the classification to two \emph{decoupling problems},
introduced in Lemma~\ref{lem:rat_sol_dcpl}, one said homogeneous, the other
one inhomogeneous. The classification will then go as follows.
For most set of steps and weightings (see
Section~\ref{sec:one-particular-case} for the one exception), and depending
on the existence of solutions to these decoupling equations, either
we will be in the case of Lemma~\ref{lem:hypertrans},
and then the generating function will not be D-algebraic in $x$ and $y$,
either we will be in the case of Lemma~\ref{lem:lift_sol},
and then we will be able to give explicit algebraic solutions for
$Q(x,y)$.

Section~\ref{sect:decoupl_eq} is thus devoted to the study of the
rational solutions to decoupling equations of the form
$\gamma_1(x(s),y(s)) f(x(s)) + \gamma_2(x(s),y(s)) g(y(s)) + c = 0$
for all $s \in \P^1$,
with fractions $\gamma_1$, $\gamma_2$ and $c$ depending
on the weights $(d_v)_{v \in \calS}$, $a$, $b$ and $t$.
Through a process called \emph{pole propagation}, we will see
that the existence of rational solutions is conditioned to the relative
positions of some particular points of $\P^1$ with respect to the action of
$\sigma$. More explicitly, the relative position between
two points is defined as
the unique integer $n=\delta(P,Q)$ such that
$\sigma^n P = Q$, which we call the \emph{$\sigma$-distance}.
In the end, we extract necessary conditions for
the existence of rational solutions to the decoupling problem, based on the
values $\delta(P,Q)$ for $(P,Q) \in \cL^-\times \cL^+$,
for some finite sets $\cL^-$ and $\cL^+$.

Section~\ref{sect:decide_diff} gives a way to compute
this $\sigma$-distance, based on the fact that we can define valuations
on the coordinates of the points that are considered in this paper (they are the orbits
of points in $\cL^-$ and $\cL^+$). These
valuations evolve with the action of $\sigma$ in a deterministic way,
allowing to effectively compute the $\sigma$-distance between two
points given a fixed weighting. We extend this algorithm to find
algebraic relations between the weights that guarantee a certain $\sigma$-distance.
In the end, we compile the computations for our case
in Appendix~\ref{sect:mat_results}.

In Section~\ref{sec:classification}, we finally exploit the
$\sigma$-distance computation of Section~\ref{sect:decide_diff} along
with the criteria determined in Section~\ref{sect:decoupl_eq} to treat
all the cases. In the end, we obtain the classification
in the form of Theorem~\ref{thm:thm_clas}.

The very last Section~\ref{sec:conclusion-comments} discusses
various questions left at the end of the present paper, related
to alternative proofs for the algebraic cases, the phase transitions
of the models, and the general study of these inhomogeneous decoupling equations.

\section{Quadrant walks and $q$-difference equations}
\label{sec:study-walks-thro}

\subsection{Quadrant walks with interacting boundaries}
\label{sec:quadrant-walks-with}
\subsubsection*{Quadrant walks}

\begin{figure}[h]
  \centering
  \includegraphics{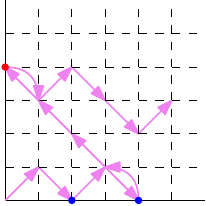}
  \caption{A walk in the quadrant with set of steps $\calS = \{(-1,1), (1,-1), (1,1)\}$,
    using 13 steps, 2 contacts with the $x$-axis
    and 1 contact with the $y$-axis.}\label{fig:2Dwalk}
\end{figure}

We recall here some basic definitions relative to the enumeration of walks
in the quadrant.
Consider a finite subset~$\calS$ of vectors in $\Z^2 \setminus \{(0,0)\}$.
A walk of $n$ steps modeled on $\calS$ (starting at $(0,0)$)
is a sequence $w = v_1, v_2, \dots, v_n$ of steps $v_i \in \calS$.
We have the additional condition that
the walk is a \emph{quadrant walk}, that is, at every index $i \le n$, we
require that both coordinates of $\sum_{j \le i} v_i$
are nonnegative. Figure~\ref{fig:2Dwalk} represents
an example of a quadrant walk.

When counting these walks, several statistics are included.
We list below the statistics of a walk $w$
of $n$ steps that we consider in this paper:
\begin{itemize}
\item the coordinates $(i,j)$ of the last point
  of $\N^2$ they visit.
\item for each step $v$ in $\calS$ the number $n_v$ of occurrences
  of this particular step in the walk $w$.
\item the number of contacts (also called interactions) of $w$ with the axes.
  Recall that a contact with the $x$-axis occurs at each $i \ge 1$ such that
$\sum_{j \le i} v_j$ is zero. Thus, performing twice the step $(1,0)$
starting from $(0,0)$ accounts for two contacts with the $x$-axis,
despite the walk remaining on the $x$-axis.  The number of contacts
with the $x$-axis (resp. $y$-axis) is denoted by $n_x$ (resp. $n_y$).
The \emph{interacting boundaries} qualification refers to this
statistics.
\end{itemize}

\subsubsection*{Weighted walks} \label{sec:weighted-walks}

In combinatorics, it is quite common to associate weights to
objects, depending on their statistics, often for probabilistic
purposes. Below, we define the weight of a quadrant walk
with interacting boundaries.

For each $v \in \calS$ is given $d_v > 0$ the real positive
weight associated with the step $v$.
Moreover, we are also given positive real numbers $a>0$ (resp. $b>0$)
the \emph{Boltzmann weight} associated with the $x$-axis (resp. $y$-axis).
The weight of the walk $w$ is then defined as the monomial \[
  \mbox{weight}(w) := \left( \prod_{v \in \calS} d_v^{n_v} \right) a^{n_x} b^{n_y}.
\]
This definition of weight is correlated with an associated probability distribution
on walks of length $n$, defined so that the probability of a walk $w$
is proportional to its weight, i.e.
\[
  \P_n(w) := \frac{\mbox{weight}(w)}{\displaystyle \sum_{\mbox{$w'$ walk of length $n$}} \mbox{weight}(w')}.
\]
The weighting induces a bias on the distribution of quadrant walks,
for instance a bigger value for $d_{v}$ increases the probability
of performing the step $v$, or a small value of $b$ favors walks
that have fewer contacts with the $y$-axis.
The above heuristics can be made precised through limit properties
of the probability $\P_n$, that define \emph{phases}
(see Section~\ref{sec:phase-transitions}).

We call a \emph{weighted model} of quadrant walks with interacting boundaries
a set of steps $\calS$ together with a \emph{weighting},
i.e. positive real weights $(d_v)_{v \in \calS}$,
and Boltzmann weights $a>0$ and $b>0$. Given a weighted model, we write
\[\Fpar \eqdef \Q( (d_v)_{v\in\calS},a,b)\] for the subfield of~$\C$ generated by the weights.

\subsubsection*{Generating function and functional equation}

Given a weighted model $(\calS,(d_v)_{v\in\calS},a,b)$,
the generating function of quadrant walks on this weighted model is defined as
\[
  Q(x,y) \eqdef \sum_{\mbox{$w$ walk}} \mbox{weight}(w) x^i y^j t^n = \sum_{w \text{ walk}} \left(\prod_{v \in \calS} d_v^{n_v}\right) a^{n_x} b^{n_y} x^i y^j t^n.
\]
Since there is a finite number of quadrant walks of length $n$
and that the walks always terminate in the first quadrant,
the generating function $Q(x,y)$ belongs to $\Fpar[x,y] \llbracket t \rrbracket$
(recall that $R \llbracket t \rrbracket$ denotes the ring of formal power series
in the variable $t$ with coefficients in the ring $R$).
Note that it is harmless to have the $d_{v}$, $a$ and $b$ as real numbers
for the exact counting of walks with regards to the
statistics $n_v$, $n_x$ and $n_y$.
As the transcendence degree of $\R$ over $\Q$ is infinite, one may choose
algebraically independent weights $d_v$, $a$ and $b$ over $\Q$, and still perform
coefficient extraction to get these statistics since the coefficient $[t^n] Q(x,y)$
belongs to $\Q[(d_v)_{v\in\calS},a,b,x,y]$. This is why we directly consider the generating
function of weighted walks, only having $x$, $y$ and $t$ as variables.


The generating function is characterized through a functional equation.
In Theorem~6 of~\cite{beaton_exact_2019}
the authors derive the following explicit
functional equation for the series~$Q(x,y)$ when the steps of the model~$\calS$ are
small (that is $\calS \subseteq \{-1,0,1\}^2$):
\begin{align*} \label{eq:eq_func_Qxy_general}
  K(x,y) Q(x,y) &= \frac{xy}{ab} + x\left(y - \frac{y}{a} - t A_{-1}(x)\right) Q(x,0) \\
  &+ y \left(x - \frac{x}{b} - t B_{-1}(y)\right) Q(0,y)
   - \left(\frac{xy}{ab}(1-a)(1-b) - t \varepsilon \right) Q(0,0). \numberthis
\end{align*}
This functional equation generalizes those found in the quadrant
walks literature for the study of weighted models
independently of the interaction statistics.
The polynomial $K(x,y) \eqdef xy (1 - t S(x,y))$ is commonly called
the \emph{kernel}, where $S(x,y) \eqdef
\sum_{(i,j) \in \calS} d_{i,j} x^i y^j$ encodes
the set of steps as a Laurent polynomial.
The fractions $A_i(x)$ and $B_i(y)$ are then defined as
$A_i(x) \eqdef [y^i] S(x,y)$ and $B_j(y) \eqdef [x^j] S(x,y)$.
Finally, according to the notation of \cite{beaton_exact_2019},
the variable~$\varepsilon$ is set to $1$ if the step $(-1,-1)$ is an element
of $\calS$, and $\varepsilon=0$ otherwise.

This functional equation is part of the class of polynomial equations
in two \emph{catalytic variables}, $x$ and $y$, as the equation
relates $Q(x,y)$ with its specializations $Q(x,0)$,
$Q(0,y)$ and $Q(0,0)$. While the theory behind polynomial equations
involving only one catalytic variable is well known (the solutions are
always algebraic, see~\cite{BMJ}), this is not at all the case for equations
of two catalytic variables. The systematic study of walks in the quadrant revolves around
such equations, which is one of the reasons why this topic is challenging.

\subsubsection*{The differential classification}
In general, we do not expect to find a closed form for~$Q(x,y)$
from this equation.
Thus, a more reasonable question is the \emph{classification}
of the generating function, that is knowing where the function~$Q(x,y)$
fits in the following \emph{algebraic-differential hierarchy}:
\[
  \mbox{rational} \subset \mbox{algebraic} \subset \mbox{D-finite} \subset \mbox{D-algebraic}.
\]
We recall that the power series $Q(x,y)$ is called \emph{rational} if it
belongs to $\K(x,y,t)$;
\emph{algebraic} if it is a solution to a polynomial equation with coefficients
in $\K(x,y,t)$;
\emph{$x$-D-finite} (resp. $y$-D-finite, $t$-D-finite) if it is a solution to
a linear differential equation with respect to the variable $x$ (resp. $y$, $t$)
with coefficients in $\K(x,y,t)$, and it is \emph{D-finite} if
it is all at once $x$, $y$ and $t$-D-finite;
\emph{$x$-D-algebraic} (resp. $y$-D-algebraic, $t$-D-algebraic)
if it is a solution to a polynomial differential equation with respect
to the variable $x$ (resp. $y$, $t$)
with coefficients in $\K(x,y,t)$,
and it is \emph{D-algebraic} if it is all at once
$x$, $y$ and $t$-D-algebraic.

The problem is to classify the models of quadrant walks according to where
their generating function $Q(x,y)$ occurs in the hierarchy, and we say that a
weighted model is of class $X$ if its generating function is of class $X$.
Such a classification gives qualitative information
on the complexity of the walk. If it is low (at most D-finite), one
can expect a nicer combinatorial interpretation, maybe a closed
form, and fast algorithms to compute the coefficients. Otherwise,
their study is more complicated, and requires specific
techniques for each equation.

When setting the weights $d_v$, $a$ and $b$ to $1$
(which amounts to ignoring the interaction and
weights statistics, and is the original
setting of the systematic classification), the classification of small
steps models was completed in 2018.  The methods of this first
classification extend when considering arbitrary positive real weights
$d_v$, and it is now complete as well (see~\cite{dreyfus2019differential}).

Some of these techniques may in turn be adapted to the study of walks with the
interacting boundaries statistics, where we allow
other values of $a$ and $b$, and the weighted models of walks with interacting
boundaries that have been studied up to now rely on the finiteness of
the \emph{group of the walk}, which is a group of birational
transformations attached to each weighted model.  This is the case
in~\cite{tabbaraExactSolutionTwo2014}, where the walks with
interacting boundaries are completely solved for one specific model
(the reversed Gessel model, also called Gouyou-Beauchamps) for all weights $a, b > 0$, establishing the full
phase diagram.  This is also the case in
\cite{beaton_quarter-plane_2021}, where the authors fully solve the
Kreweras and reverse Kreweras walks with interaction for any value of
the Boltzmann weights.  Finally, in~\cite{beaton_exact_2019}, the
authors systematically investigate the models having a finite group,
for some Boltzmann weights, mainly $(a,a)$, $(1,b)$, $(a,1)$ and
$(a,b)$ for $a$ and $b$ algebraically independent over $\Q$, giving
upper bounds on the complexity of the generating function $Q(1,1)$.
We propose to treat a case with an infinite group, and for nongeneric
weights $d_{i,j}$, $a$ and $b$.

\subsection{Genus zero models}\label{sec:restr_models}

The authors of~\cite{DreyfusHardouinRoquesSingerGenuszero2}
define five models called the \emph{genus zero models} (the terminology
is explained in the next paragraph).
They are listed in Figure~\ref{fig:gen0_supports} below,
and they will be referenced in this paper as $\calS_1$, $\calS_2$, etc.
\begin{figure}[h!]
  \centering
  \includegraphics[width=.7\textwidth]{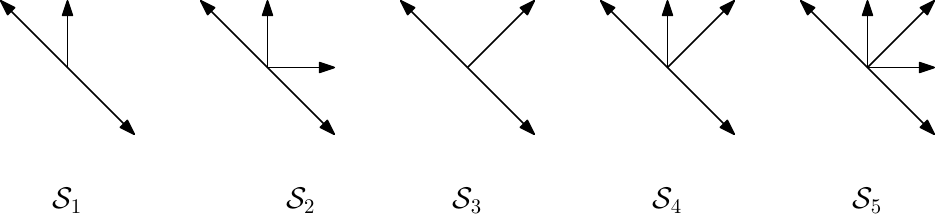}
\caption{The five models of genus~$0$} \label{fig:gen0_supports}
\end{figure}
The goal of the current paper is to establish the full classification
of walks with interacting boundaries based on these sets of steps.

For these sets of steps, we may perform simplifications on the general functional
equation~(\ref{eq:eq_func_Qxy_general}).
First, the Laurent polynomial is of the form
\[
  S(x,y) = d_{1,-1} \tfrac{x}{y} + d_{-1,1} \tfrac{y}{x} + d_{1,0} x + d_{0,1} y + d_{1,1} x y,
\]
with $d_{1,-1}$
and $d_{-1,1}$ always nonzero, and at least one of the $d_{0,1}$, $d_{1,0}$
or $d_{1,1}$ nonzero. Moreover, we have that
$A_{-1}(x) = d_{1,-1} x$ and $B_{-1}(y) = d_{-1,1} y$,
and since no model contains the step $(-1,-1)$,
the variable~$\varepsilon$ is always zero. Finally, for
every model of Figure~\ref{fig:gen0_supports},
the power series $Q(0,0)$ is equal to $1$.
Indeed, any nontrivial walk one of these models has an ending point
$(i,j)$ satisfying $i+j > 0$ by an easy induction. Hence setting both
$x$ and $y$ to $0$ in $Q(x,y)$ leaves only the term in $t^0$, which
is equal to $1$. Summarizing these simplifications, the
functional equation~(\ref{eq:eq_func_Qxy_general}) can be rewritten as
\begin{align*}
  K(x,y) Q(x,y) &= \frac{x y}{ab} \left(a+b-ab\right) \\
                &+ x \left(y - \frac{y}{a} - t d_{1,-1} x\right) Q(x,0)
                  + y \left(x - \frac{x}{b} - t d_{-1,1} y \right) Q(0,y).
\end{align*}
We now introduce the following notations:
\begin{equation}\label{notat:eq}
  \begin{aligned}
    A &\eqdef 1 - \frac{1}{a},
    &B &\eqdef 1 - \frac{1}{b},
    &\omega &\eqdef \frac{1}{ab}(a+b-ab), \\
    \gamma_1(x,y) &\eqdef \frac{A}{x} - \frac{t d_{1,-1}}{y},
    &\gamma_2(x,y) &\eqdef \frac{B}{y} - \frac{t d_{-1,1}}{x},
    &\gamma(x,y) &\eqdef \frac{\gamma_1(x,y)}{\gamma_2(x,y)}.
  \end{aligned}
\end{equation}
With these notations, the functional equation can finally be rewritten as
\begin{align} \label{eq:eq_func_Qxy_gen0}
  K(x,y) Q(x,y) &= \omega x y + x^2 y \gamma_1(x,y) Q(x,0) + x y^2 \gamma_2(x,y) Q(0,y).
\end{align}

In the remaining of the section, we show how to further exploit the particularity
of the genus~$0$ models to study (\ref{eq:eq_func_Qxy_gen0}), and in the end
classify $Q(x,y)$ for any weighting on one of these set of steps.

\subsubsection*{The kernel curve}

Given a weighted model, consider the kernel polynomial $K(x,y,t)$,
which belong to $\Q(d_{i,j})[x,y,t]$. For
every complex number~$t$, this polynomial defines a complex affine
curve in $\C^2$ defined as
\begin{align*}
  \Et &\eqdef \{ (x,y) \in \C \times \C \st K(x,y) = 0 \}.
\end{align*}
One considers the projective completion of this curve in $\P^1 \times \P^1$.
Here, $\P^1 = \P^1(\C)$ designates the complex projective line,
which may be defined as the quotient of $\C \times \C$ by
the equivalence relation \[
  (x_0,x_1) \sim (x'_0, x'_1) \iff \exists \lambda \in \C^{\star}, (\lambda x_0, \lambda x_1) = (x'_0,x'_1).\]
One denotes by $[x_0:x_1]$ the class of $(x_0,x_1)$. The set $\C$ embeds into $\P^1$
through the map $x \mapsto [x:1]$, and we furthermore denote $\infty := [1:0]$.
The projective
completion of $E_t$ in $\P^1 \times \P^1$ is then
\begin{align*}
  \Etproj &\eqdef \{ ([x_0:x_1],[y_0:y_1]) \in \P^1\times\P^1 \st x_1^2 y_1^2 K(\tfrac{x_0}{x_1}, \tfrac{y_0}{y_1}) = 0 \}.
\end{align*}

In~\cite{DreyfusHardouinRoquesSingerGenuszero2}
the authors examine the $\Etproj$ that occur 
for weighted models with small steps.
They prove that apart from trivial cases, the curve
$\Etproj$ is irreducible. When it is irreducible, they show that it
is either nonsingular of genus~$1$
or of genus~$0$ with a unique singular point~$\Omega = ([0:1],[0:1])$.
By abuse of notation, a weighted model is said to be of genus~$g$ if the curve
$\Etproj$ has genus~$g$.

It turns out that to study the models of genus~$0$,
it is enough to consider the five fundamental sets of steps
of Figure~\ref{fig:gen0_supports}, hence the qualification
of genus zero models.

\subsubsection*{Group and parametrization of the kernel curve}
When the kernel curve has genus~$0$, the authors of
\cite{DreyfusHardouinRoquesSingerGenuszero2}
construct a specific rational parametrization $\phi : \P^1 \rightarrow \Etproj$.
We summarize basic
facts and vocabulary on this parametrization, following
Section~4.1 of~\cite{DreyfusHardouinRoquesSingerGenuszero2}.

We first recall basic facts about the \emph{group of the walk}.
It has been well established since the beginning of the study
of quadrant walks in~\cite{FIM} or~\cite{BMM} that for any model of walk with small
steps, the projective curve~$\Etproj$ is equipped with two involutive
automorphisms $\iota_1$ and $\iota_2$. They
have the property that for all $(u,v) \in \P^1 \times \P^1$,
$\iota_1 (u,v) = (u,v')$ for some $v'$ and $\iota_2 (u,v) = (u',v)$
for some $u'$. For the models of genus zero,
their expression simplifies as follows:
\begin{align} \label{eq:def_i1}
  \iota_1([1:x_1],[1:y_1]) & = \left([1:x_1],\left[1:\frac{d_{-1,1} x_1^2 + d_{0,1} x_1 + d_{1,1}}{d_{1,-1} y_1}\right]\right), \\
  \label{eq:def_i2}\iota_2([1:x_1],[1:y_1]) & = \left(\left[1:\frac{d_{1,-1} y_1^2 + d_{1,0} y_1 + d_{1,1}}{d_{-1,1} x_1}\right],[1:y_1]\right).
\end{align}
Note that we choose to write them for points of $\P^1 \times \P^1$
written in homogeneous coordinates $([1:x_1],[1:y_1])$
for reasons detailed in Section~\ref{sect:decide_diff}.

The group of the walk is then defined as the group of automorphisms
of $\Etproj$ generated
by $\iota_1$ and $\iota_2$.
These two involutions induce an automorphism $\sigma$
of $\Etproj$ defined as \[\sigma
\eqdef \iota_2 \circ \iota_1.\]

As a side note, if $\tau$ is an automorphism of $\P^1$ and
$h(s) \in \C(s)$, then we will write $h^{\tau}(s) \eqdef h(\tau(s))$.
The reason for the exponential notation is because the composition
action of automorphisms of $\P^1$ on the function field of $\P^1$ is a right action,
so that $h^{\tau_1 \tau_2} = \left(h^{\tau_1}\right)^{\tau_2}$.

We are now going summarize the properties of the parametrization $\phi$ of~$\Etproj$
which was defined in~\cite{DreyfusHardouinRoquesSingerGenuszero}, and that we will use in the present paper.
It is constructed so that
the action of the group of the walk lifts through $\phi$ in a nice way.

We recall basic facts on \emph{divisors} of a function
field of an algebraic curve. We will be even more specific, and restrict to $\P^1$,
whose function field is $\C(s)$ with $s$ transcendental. A comprehensive
introduction to these notions is contained in~\cite{Stichtenothalgfunctionfields}.
\begin{defprop}[Chapter 1 of \cite{Stichtenothalgfunctionfields}] \label{prop:recap_divisors}
    A \emph{divisor} is a formal finite sum of points $D = \sum_{P \in \P^1} n_P P$ where
    $n_P$ are integers.
    The \emph{degree} of a divisor $D = \sum_{P} n_P P$ of $\P^1$ is defined
    as $\deg D = \sum_P n_P$. The map $D \mapsto \deg D$ is a group homomorphism.
    The following properties hold:
    \begin{enumerate}
    \item Let $h$ be a nonzero function in $\C(s)$.
      The function $h$ has finitely many zeros in $\P^1$.
      The \emph{zero divisor} of $h$ is thus defined as
      \[
        (h)_0 = \sum_{\mbox{$P$ zero of $h$}} \ord_P(h) \cdot P
      \]
      where $\ord_P(h)$ is the multiplicity of $P$ as a zero of $h$.
      Similarly, the \emph{polar divisor} $(h)_{\infty}$ of $h$ is defined as
      the zero divisor of $h^{-1}$.

    \item The \emph{principal divisor} associated to a nonzero function $h$ is defined as
      \[
        (h) = (h)_0 - (h)_{\infty}.
      \]
      It has the property that $(h) = 0$ if and only if $h \in \C$.
    \item For $u$, $v$ two nonzero functions in $\C(s)$,
      then $(1/u) = -(u)$ and $(uv) = (u) + (v)$.
    \item For $h \not\in \C$, the following holds:
      \[
        1 \le \deg (h)_0 = \deg (h)_{\infty} = [\C(s) : \C(h(s))] < \infty
      \]
      (recall that for $A$ an extension,
      $[A:k]$ denotes the dimension of $A$ as a $k$-vector space).
    \end{enumerate}
\end{defprop}

We may now state some properties of the parametrization $\P^1 \rightarrow \Etproj$.
\begin{proposition}[Section~4.1 of \cite{DreyfusHardouinRoquesSingerGenuszero2}] \label{prop:recap_gen0}
  For any model of Figure~\ref{fig:gen0_supports} and weighting $d_v$,
  and any real number $t \in (0,1)$
  transcendental over $\Q(d_{i,j},a,b)$, the following assertions hold.
  \begin{enumerate}
  \item There exists a
    rational parametrization
    $\phi : s \mapsto (x(s),y(s))$ of $\P^1$ onto $\Etproj$.
    The fractions $x(s)$ and $y(s)$
    both belong to $\overline{\Fpar(t)}(s)$, where
    $\overline{\Fpar(t)}$ is the algebraic closure of $\Fpar(t)$.
  \item The parametrization $\phi$
    is one-to-one everywhere except for $\phi(0) = \phi(\infty) = \Omega$,
    with $\Omega \eqdef  ([0:1],[0:1])$,
    i.e. $(x(0),y(0)) = (x(\infty),y(\infty)) = \Omega$.
  \item The divisors (Proposition~\ref{prop:recap_divisors})
    of the functions $x\eqdef x(s)$ and $y\eqdef y(s)$ on the curve $\P^1$ are
    \begin{align*}
      (x) = 0 + \infty - Q_1 - Q_2, &  & (y) = 0 + \infty - Q_3 - Q_4,
    \end{align*}
    for some points $Q_i \neq 0, \infty$ of $\P^1$ (that do not have to be distinct).

  \item The group lifts in the following way.
    There exists a real number $q \not\in \{-1,1\}$,
    with $q$ algebraic over $\Fpar(t)$, such that for all $s \in \P^1$ one has
    \[\iota_1(\phi(s)) = \phi(\tfrac{1}{s})
      \mbox{ and } \iota_2(\phi(s)) = \phi(\tfrac{q}{s}).\]
    When the context is clear, we will also denote by $\iota_1$,
    $\iota_2$ and $\sigma$ the automorphisms on $\P^1$ defined by
    \begin{align*}
      \iota_1(s) &\eqdef \tfrac{1}{s}
      & \iota_2(s) &\eqdef \tfrac{q}{s}
      & \sigma(s) &\eqdef qs.
    \end{align*}
    As the multiplicative order of $q$ is infinite, so is the order
    of $\sigma$. Moreover, the only points of $\P^1$ whose orbit
    under the action of $\sigma$ is finite are $0$ and $\infty$.
  \end{enumerate}
\end{proposition}

\begin{proposition}[Section~1.3 of \cite{DreyfusHardouinRoquesSingerGenuszero}] \label{prop:recap_gal}
  The function field $\C(s)/\C$ of $\P^1$ has
  the following lattice.
  \begin{center}
    \includegraphics[height=5cm]{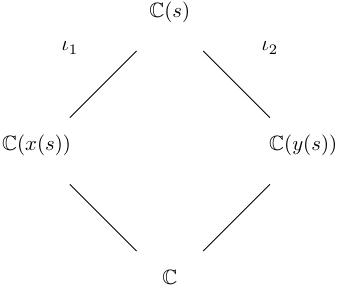}
  \end{center}
  \begin{enumerate}
  \item The extension $\C(s) / \C(x(s))$ is Galois of degree $2$,
    with Galois group generated by the involution $\iota_1$.
    This means that if $h = h^{\iota_1}$, then $h \in \C(x(s))$.
  \item Similarly, the extension $\C(s) / \C(y(s))$ is Galois of degree $2$,
    with Galois group generated by the involution $\iota_2$.
  \item The field of constants $\C$ is the intersection $\C(x(s)) \cap \C(y(s))$,
    and as a result it is also the subfield of functions fixed
    by $\iota_1$ and $\iota_2$.
    Moreover, if $f(s)$ is fixed by $\sigma$, then $f(s)$ belongs to $\C$.
  \end{enumerate}
\end{proposition}

\subsection{The \texorpdfstring{$\boldsymbol{q}$}{q}-difference equations}
\label{sec:qdiff-equ}
We now fix the setting in which we can evaluate the functional
equation~(\ref{eq:eq_func_Qxy_gen0}) for $Q(x,y)$
on the curve $\Etproj$ through the parametrization $\phi$,
in the manner of Section~2.1 of~\cite{DreyfusHardouinRoquesSingerGenuszero2}.
Once the conditions of this evaluation are fixed, this
will transform the catalytic equation on $Q(x,y)$ (a power series in $x$ and $y$)
into a functional equation relating $Q(x(s),0)$ and $Q(0,y(s))$
(meromorphic functions over~$\C~\subset~\P^1$).
The symmetries of $x(s)$ and $y(s)$ with respect
to the group (Proposition~\ref{prop:recap_gen0})
will then allow us to construct a
functional equation on $Q(x(s),0)$ only.

We first study the divisors of the functions
\begin{equation}  \label{eq:20}
  \begin{aligned}
  \gamtld_1(s) &\eqdef \gamma_1(x(s),y(s))
  & \gamtld_2(s) &\eqdef \gamma_2(x(s),y(s))
  \end{aligned}
\end{equation}
on $\P^1$ through a routine computation.
Recall that
\begin{align*}
  \gamma_1(x,y) &= \frac{A}{x} - t \frac{d_{1,-1}}{y}
  & \gamma_2(x,y) &= \frac{B}{y} - t \frac{d_{-1,1}}{x},
\end{align*}
as defined in (\ref{notat:eq}).
\begin{proposition} \label{prop:div_gamma}
  For any real number $t$ defined as in Proposition~\ref{prop:recap_gen0},
  the extensions $\C(s) / \C(\gamtld_1)$ and $\C(s) / \C(\gamtld_2)$
  have degree two, and the functions $\gamtld_1$ and $\gamtld_2$ have
  the following divisors:
  \begin{align*}
    (\gamtld_1) &= P_1 + P_2 - 0 - \infty,
    & (\gamtld_2) &= P_3 + P_4 - 0 - \infty.
  \end{align*}
  The points $P_1$, \dots, $P_4$ are
  distinct from $0$ and~$\infty$.
\end{proposition}
\begin{proof}
  In this proof, we write
  $x = x(s)$ and $y = y(s)$, these functions thus
  satisfying $K(x,y) = 0$ for $K(X,Y)$ the kernel polynomial.
  We only perform the proof for $\gamtld_1 = \frac{A}{x}-\frac{t d_{1,-1}}{y}$,
  the study of $\gamtld_2$ being symmetric.
  The computations can be followed in the Maple worksheet.
  We are going to prove that the minimal polynomial
  of $x$ over $\C(\gamtld_1)$ has degree $2$, thus showing that $[\C(x):\C(\gamtld_1)] = 2$.
  We will then prove that $\gamtld_1$ has the announced poles $0$ and $\infty$.

  We first produce a vanishing polynomial of $x$ over $\C(\gamtld_1)$.
  By definition, the polynomial $K(X,y) = 0$ is a vanishing polynomial
  of $x$ over $\C(y)$. Hence, expressing $y$ in terms of $x$ and $\gamtld_1$
  (which we can do since $t d_{1,-1}$
  is nonzero), we are left
  to consider the polynomial
  \begin{align*}
    P(X) &=  (d_{1,1} d_{1,-1} t^2 - d_{1,0} t \gamtld_1 + \gamtld_1^2) X^2 \\
         &+ (d_{0,1} d_{1,-1} t^2 + A d_{1,0} t + (1 - 2 A) \gamtld_1) X \\
         &+ d_{1,-1} d_{-1,1} t^2 + A^2 - A,
  \end{align*}
  which by construction is a vanishing polynomial of $x$
  with coefficients in $\C[\gamtld_1] \subset \C(\gamtld_1)$.
  Moreover, $P(X)$ is nonzero, for its constant coefficient
  $d_{1,-1} d_{-1,1} t^2 + A^2- A \in \Fpar[t]$ is nonzero. Indeed,
  the real number $t$ is transcendental over $\Fpar$ and the coefficient
  $d_{1,-1} d_{-1,1}$ of $t$ is nonzero for all genus zero models
  (Figure~\ref{fig:gen0_supports}).
  Hence, $x$ is algebraic over $\C(\gamtld_1)$.

  We now show that $P(X)$ is irreducible
  in $\C(\gamtld_1)[X]$.
  First note that the leading coefficient
  $d_{1,1} d_{1,-1} t^2 - d_{1,0} t \gamtld_1 + \gamtld_1^2$
  of $P(X)$ is nonzero. Indeed, it is a nonzero polynomial of $\C[\gamtld_1]$,
  with $\gamtld_1$ transcendental over $\C$
  (the function $x$ is both transcendental over $\C$ and algebraic over $\C(\gamtld_1)$).
  Hence, $P(X)$ is a degree two polynomial.
  In order to show that it is irreducible, we thus compute its
  discriminant $\Delta \in \C(\gamtld_1)$
  and show that it cannot be a square in $\C(\gamtld_1)$.
  The discriminant is expressed as follows:
  \begin{align*}
    \Delta &= (-4 d_{1,-1} d_{-1,1} t^2 + 1) \gamtld_1^2 \\
           &+ (4 d_{1,0} d_{1,-1} d_{-1,1} t^3 - 4 A d_{0,1} d_{1,-1} t^2 + 2 d_{0,1} d_{1,-1} t^2 - 2 A d_{1,0} t) \gamtld_1 \\
           &+ (d_{0,1}^2 d_{1,-1}^2 t^4 - 4 d_{1,1} d_{1,-1}^2 d_{-1,1} t^4 + 2 A d_{0,1} d_{1,0} d_{1,-1} t^3 \\
           &\quad + A^2 d_{1,0}^2 t^2 - 4 A^2 d_{1,1} d_{1,-1} t^2 + 4 A d_{1,1} d_{1,-1} t^2).
  \end{align*}
  We see that $\Delta$ belongs to $\C[\gamtld_1]$, hence $\Delta$ is a square
  in $\C(\gamtld_1)$ if and only if $\Delta$ is a square in $\C[\gamtld_1]$.
  In turn, as $\Delta$ has degree two as a polynomial
  in $\gamtld_1$ (the leading coefficient
  $-4 d_{1,-1} d_{-1,1} t^2 + 1$ is nonzero since $t$ is transcendental
  over $\Fpar$), $\Delta$ is a square if
  and only if its discriminant $\delta$ with respect to the variable
  $\gamtld_1$ is zero.
  The discriminant $\delta$ factors in $\Q[d_{i,j},A][t]$ into
  $\delta = 16 t^2 f_1 f_2 f_3$, the factors $f_i$ written in the table below.
  \begin{table}[!h]
    \centering
    \begin{tabular}{|c|c|c|}\hline
      $f_1$ & $d_{1,-1}$ \\ \hline
      $f_2$ & $d_{1,-1} d_{-1,1} t^2 + A(A-1)$ \\\hline
      $f_3$ & $(d_{0,1}^2 d_{1,-1}  + d_{1,0}^2 d_{-1,1}  - 4 d_{1,1} d_{1,-1} d_{-1,1}) t^2 + d_{0,1} d_{1,0} t + d_{1,1}$ \\\hline
    \end{tabular}
  \end{table}

  As $t$ is transcendental over $\Fpar$,
  the polynomial $\delta$
  is zero if and only if one $f_i$ is zero
  if and only if all coefficients of one $f_i$ viewed as
  a polynomial in $\Fpar[t]$ are zero.
  Define an ideal $I$ of $\Z[d_{i,j},A]$ by
  $I = d_{1,-1} d_{-1,1} (d_{1,1}, d_{0,1}, d_{1,0})$,
  and for $i \in \{1,2,3\}$ the ideal $J_i$ of $\Z[d_{i,j},A]$
  generated by the coefficients in $t$
  of $f_i$. Then one sees that $I \subset J_1$,
  $I \subset J_2$, and finally through elimination
  that $I^3 \subset J_3$. Hence,
  if $\delta$ is zero, then
  $(d_{1,-1}, d_{-1,1}, d_{0,1}, d_{1,0}, d_{1,1}, A)$ must satisfy
  either $d_{1,-1} = 0$, or $d_{-1,1} = 0$,
  or $d_{1,0} = d_{0,1} = d_{1,1} = 0$, which is
  never the case given the constraints
  on the supports (Figure~\ref{fig:gen0_supports}).
  Therefore, $\delta$ is always nonzero,
  $\Delta$ is never a square, and we conclude that $P(X)$ is
  always irreducible in $\C(\gamtld_1)[X]$.
  Thus, since $P$ is an irreducible vanishing polynomial of $x$,
  and since $\C(x,\gamtld_1) = \C(x,y)$, this proves that $[\C(x,y):\C(\gamtld_1)]
  = [\C(x,\gamtld_1):\C(\gamtld_1)] = 2$.

  We now conclude on the poles of $\gamtld_1$.
  By Proposition~\ref{prop:recap_gen0}, the divisors of $x$ and $y$ are respectively
  $(x) = 0 + \infty - Q_1 - Q_2$ and $(y) = 0 + \infty - Q_3 - Q_4$ for
  some points $Q_i$.
  Hence, the poles of $\gamtld_1 = \tfrac{A}{x} - t \tfrac{d_{1,-1}}{y}$
  are either $0$ or $\infty$, which both have order at most $1$.
  But by Proposition~\ref{prop:recap_divisors},
  $\deg (\gamtld_1)_{\infty} = [\C(s) : \C(\gamtld_1)] = 2$,
  thus we conclude that $\gamtld_1$ has these two poles,
  and thus $(\gamtld_1)_{\infty} = 0 + \infty$.
\end{proof}

We now determine for which real numbers $t$ the composition of the functional
equation~(\ref{eq:eq_func_Qxy_gen0}) with $(x(s),y(s))$ (which both
depend on $t$) is well defined.
\begin{proposition} \label{prop:holom_eq}
  There exists a positive real number $r > 0$ such that for
  every real number $t < r$ transcendental over $\Fpar$,
  there exist two open sets $U_0$ and $U_{\infty}$ of $\P^1$ such that $0 \in U_0$,
  $\infty \in U_{\infty}$, and so that the functions $Q(x(s),y(s))$,
  $Q(x(s),0)$, $Q(0,y(s))$,
  are analytic on $U_0 \cup
  U_{\infty}$.
  Moreover, there exists an open set $V$ satisfying $0 \in V \subset U_0$ such that
  $\iota_2(V) \subset U_{\infty}$ and $\sigma^{-1}(V)~\subset~U_{0}$.
\end{proposition}
\begin{proof}
  We are going to show that for $t$ small enough,
  the series $Q(x,y) \in \C[x,y]\llbracket t\rrbracket$ is convergent on $\{(x,y) \in \C \st |x|, |y| < 1\}$.
  Let $|x|, |y| < 1$, and $M \eqdef  \sup~\{|d_{i,j}|, |a|, |b|\}$ (note that $M$ is positive).
  Then each walk $w$ of length $n$
  contributes to a value of norm at most $M^{2n}$
  to the coefficient of $t^n$ in $Q(x,y)$. Indeed,
  each performed step
  involves at most two weights (one $d_{i,j}$ and possibly one additional $a$ or $b$).
  As the steps are small, the coefficient corresponding to $w$ has norm at
  most $M^{2n} |x|^n |y|^n \le M^{2n}$, for $|x|, |y| < 1$.
  Moreover, the set of steps of any
  of the considered weighted models is finite of cardinal
  at most $5$ (Figure \ref{fig:gen0_supports}),
  hence the coefficient of $t^n$ in $Q(x,y)$ has norm of at most
  $\left(5 M^2\right)^n$.
  Hence, when $|x|, |y| < 1$, the
  power series $Q(x,y) \in \C(x,y)\llbracket t\rrbracket$ has a positive radius of convergence $\rho \ge \frac{1}{5 M^2}$.
  We thus fix a real number $t$ so that $0 < t < \rho$
  and $t$ is transcendental over the field $\Fpar$.
  Since
  the coefficients in $t$ of $Q(x,y)$ are polynomials in $x$ and $y$,
  the function $Q(x,y)$ is analytic in $x$ and $y$ at
  $(0,0)$.

  We now study the convergence of the composition of
  the functions appearing in (\ref{eq:eq_func_Qxy_gen0})
  with the parametrization $\phi(s) = (x(s),y(s))$.
  First,
  the functions $x(s)$ and $y(s)$ belong to $\C(s)$
  with $x(0) = y(0) = x(\infty) = y(\infty) = 0$,
  thus they are both analytic at the points $0$ and $\infty$.
  Thus, by composition, the functions
  $Q(x(s),y(s))$, $Q(x(s),0)$ and $Q(0,y(s))$ are analytic at $0$ and
  $\infty$.
  This proves the existence of the two announced open sets
  $U_0 \ni 0$ and $U_{\infty} \ni \infty$.

  Finally, we construct $V$ to be $U_0 \cap \iota_2^{-1}(U_{\infty})
  \cap \sigma(U_0)$. This is open because $\iota_2(s) = \frac{q}{s}$
  and $\sigma(s) = q s$, which are continuous functions
  in $\P^1 \rightarrow \P^1$.
  Moreover, $V$ contains $0$ because $\iota_2(\infty) = \iota_1(\infty) =
  \sigma(0) = \sigma^{-1}(0) = 0$.
\end{proof}

We now fix some small enough $t$ transcendental over $\Fpar$
prescribed by Proposition~\ref{prop:holom_eq},
and a parametrization $\phi$ accordingly.
The evaluation of (\ref{eq:eq_func_Qxy_gen0}) on
$(x(s),y(s))$ for $s$ in $U_0 \cup U_{\infty}$
is thus well defined, yielding after dividing by $x(s) y(s)$ the
following equation of meromorphic functions on~$U_{0} \cup U_{\infty}$:
\begin{equation}  \label{eq:func_curve}
  0 = \omega + x(s) \gamtld_1(s) Q(x(s),0) + y(s) \gamtld_2(s) Q(0,y(s)).
\end{equation}
Define two analytic functions on the open set $V$ by
\begin{align*}
  \Fcap(s) &\eqdef  x(s) Q(x(s),0), & \Gcap(s) &\eqdef  y(s) Q(0,y(s)).
\end{align*}

We now use (\ref{eq:func_curve}) to
construct meromorphic continuations $\Ftld$ of $\Fcap$ and $\Gtld$ of $\Gcap$
to the whole complex plane~$\C$, by showing
like in \cite{DreyfusHardouinRoquesSingerGenuszero} that $\Fcap$ satisfies a $q$-difference equation.
We introduce the function
\begin{equation} \label{eq:intro_gamma}
  \gamtld(s) \eqdef \gamma(x(s),y(s)) = \frac{\gamtld_1(s)}{\gamtld_2(s)},
\end{equation}
where $\gamtld_1$ and $\gamtld_2$ were introduced in (\ref{eq:20}),
and we rewrite (\ref{eq:func_curve}) as follows for $s$ in $V$.
\begin{equation} \label{eq:interm_1}
  -y(s) Q(0,y(s)) = \frac{\omega}{\gamtld_2(s)} + \gamtld(s) x(s)Q(x(s),0).
\end{equation}

By Proposition~\ref{prop:holom_eq}, if $s$
is in $V$ then $\frac{q}{s} = \iota_2(s) \in U_{\infty}$,
hence from (\ref{eq:func_curve}) we also have the following equation
for $s$ in $V$:
\begin{equation} \label{eq:interm_2}
  -y(\tfrac{q}{s}) Q(0,y(\tfrac{q}{s})) = \frac{\omega}{\gamtld_2(\tfrac{q}{s})} + \gamtld(\tfrac{q}{s})
  x(\tfrac{q}{s})Q(x\left(\tfrac{q}{s}\right),0).
\end{equation}

We now use the symmetries of the functions $x(s)$ and $y(s)$.
For all $s$ in $V$, we have by Proposition~\ref{prop:recap_gen0}
that $y(\tfrac{q}{s}) = y(s)$,
hence $-y(\tfrac{q}{s})Q(0,y(\tfrac{q}{s})) = -y(s) Q(0,y(s))$. Moreover,
we also have $x(\tfrac{q}{s}) = x(\tfrac{s}{q})$,
hence $x(\tfrac{q}{s}) Q(x(\tfrac{q}{s}),0)
= x(\tfrac{s}{q})Q(x(\tfrac{s}{q}),0)$.
Finally, as $s$ is in $V$, the complex number $\tfrac{s}{q} = \sigma^{-1}(s)$ is also in
$U_0$ by Proposition~\ref{prop:holom_eq}, so
we can replace $x(\tfrac{q}{s}) Q(x(\tfrac{q}{s}),0)$
with $\Fcap(\tfrac{s}{q})$ in (\ref{eq:interm_2}).
Hence, eliminating $y(s) Q(0,y(s))$ between
(\ref{eq:interm_1}) and (\ref{eq:interm_2}) yields
the following $q$-difference equation on $\Fcap(s)$ for all $s$ in $V$:
\begin{equation*}
  \Fcap(\tfrac{s}{q}) = \frac{\gamtld}{\gamtld^{\iota_2}}(s) \Fcap(s) + \left(\frac{\omega}{\gamtld_2(s)}
    - \frac{\omega}{\gamtld_2^{\iota_2}(s)}\right) \frac{1}{\gamtld^{\iota_2}(s)}.
\end{equation*}
As the absolute value of $q$ is not equal to $1$ (Proposition~\ref{prop:recap_gen0}), this functional
equation allows us to construct a unique continuation $\Ftld$ of $\Fcap$
meromorphic on the whole complex plane $\C$, satisfying the same
equation:
\begin{equation} \label{eq:f_qdiff}
  \Ftld(\tfrac{s}{q}) = \frac{\gamtld}{\gamtld^{\iota_2}}(s) \Ftld(s) + \left(\frac{\omega}{\gamtld_2(s)}
    - \frac{\omega}{\gamtld_2^{\iota_2}(s)}\right) \frac{1}{\gamtld^{\iota_2}(s)}.
\end{equation}
Indeed, assuming that $\Ftld$ is a meromorphic continuation of $\Fcap$ on some open set $U$,
the functional equation (\ref{eq:f_qdiff})
relates $\Ftld(s)$ with $\Ftld(\tfrac{s}{q})$
over $\C(s)$, which
allows us to extend uniquely $\Ftld$
as a meromorphic function on $qU \cup U \cup q^{-1} U$. As $|q| \neq 1$ (Proposition~\ref{prop:recap_gen0}), we have that
$\bigcup_{n \in \Z} q^n U = \C$, hence this process
gives a unique meromorphic continuation $\Ftld$ of $\Fcap$ on $\C$.

Now, the functions $\Fcap$ and $\Gcap$
satisfy the linear relation (\ref{eq:func_curve}) over $\C(s)$.
This relation provides
a unique meromorphic continuation $\Gtld$ of $\Gcap$ to $\C$
such that $\Ftld$ and $\Gtld$ satisfy
\begin{align} \label{eq:lin_rel_ftld_gtld}
  \gamtld_1(s) \Ftld(s) + \gamtld_2(s) \Gtld(s) + \omega &= 0.
\end{align}

Finally, from (\ref{eq:f_qdiff}) and (\ref{eq:lin_rel_ftld_gtld}),
it is easy to see that the function $\Gtld$ satisfies
the following $q$-difference equation:
\begin{equation} \label{eq:g_qdiff}
  \Gtld(qs) = \frac{\gamtld^{\iota_1}}{\gamtld}(s) \Gtld(s) + \left(\frac{\omega}{\gamtld_1(s)}
    - \frac{\omega}{\gamtld_1^{\iota_1}(s)}\right) \gamtld^{\iota_1}(s).
\end{equation}

\subsection{The D-algebraicity of \texorpdfstring{$\boldsymbol{Q(x,y)}$, $\boldsymbol{\Ftld(s)}$ and $\boldsymbol{\Gtld(s)}$}{$Q(x,y)$, $F(s)$, $G(s)$}}

The algebraic-differential properties of the
formal power series $Q(x,y)$, $Q(x,0)$ and $Q(0,y)$ and
their meromorphic counterparts $\Ftld(s)$ and $\Gtld(s)$ are related.
The following proposition relates the $x$ and $y$-D-algebraicity
of $Q(x,y)$ over $\C(x,y)$ with the $s$-D-algebraicity of $\Ftld(s)$ and $\Gtld(s)$
over $\C(s)$.
As $t$ is a fixed real number, the study of the $t$-D-algebraicity
of $Q(x,y,t)$ is not easily related to the properties of $\Ftld(s)$ and
$\Gtld(s)$, and
rigid parametrizations
are needed (see \cite{dreyfusLengthDerivativeGenerating2021}), which
are not implemented in this paper.

\begin{proposition} \label{prop:Qx_Fs}
  For $t > 0$ as in Proposition~\ref{prop:holom_eq},
  and a weighting $(d_v)_v$, $a$ and $b$, the following
  statements are equivalent:
  \begin{enumerate}[style=unboxed,align=CenterWithParen,labelwidth=\mylongest]
  \item[a] $Q(x,0)$ is $x$-D-algebraic,
  \item[a'] $Q(0,y)$ is $y$-D-algebraic,
  \item[b] $\Ftld(s)$ is $s$-D-algebraic,
  \item[b'] $\Gtld(s)$ is $s$-D-algebraic,
  \item[c] $Q(x,y)$ is $x$-D-algebraic for all $y$,
  \item[c'] $Q(x,y)$ is $y$-D-algebraic for all $x$.
  \end{enumerate}
\end{proposition}
\begin{proof}
  We recall that on the open set $V$, we have
  \begin{align*}
    \Ftld(s) &= x(s) Q(x(s),0), & \Gtld(s) &= y(s) Q(0,y(s)).
  \end{align*}
  Up to a restriction of $V$, the maps $x(s)$ and $y(s)$ are
  biholomorphisms on $V$.
  The equivalence between $(a)$ and $(b)$ (resp. $(a')$ and $(b')$)
  now follows from Lemmas~6.3 and~6.4 of~\cite{DHRS}.

  Moreover, the functions $\Ftld$ and $\Gtld$ are linearly
  related over $\C(s)$ by Equation~(\ref{eq:lin_rel_ftld_gtld}), which shows the equivalence
  between $(b)$ and $(b')$,
  and thus the equivalence between $(a)$, $(a')$, $(b)$ and $(b')$.

  Finally, $(c)$ is equivalent to $(a)$ from Equation~(\ref{eq:eq_func_Qxy_gen0}).
  Indeed, $\partial_x Q(0,y) = 0$, so $Q(0,y)$ is $x$-D-algebraic. Hence,
  since $x^2 y \gamma_1(x,y)$ and $K(x,y)$ are both nonzero elements
  of $\C(x,y)$, the closure
  properties of the $x$-D-algebraic class imply that $Q(x,y)$
  is $x$-D-algebraic if and only if $Q(x,0)$ is.
  Similarly, $(c')$ is equivalent to $(a')$.
\end{proof}

Therefore, determining the algebric-differential nature of the function $Q(x,y)$
is equivalent to determining the differential nature of either $\Ftld(s)$
or $\Gtld(s)$, which satisfy functional equations with more structure:
$q$-difference equations. Equations (\ref{eq:f_qdiff})
and~(\ref{eq:g_qdiff}) do not completely characterize $\Ftld(s)$
and $\Gtld(s)$, so we will use in a crucial way the fact that
they continue the functions $x(s) Q(x(s),0)$ and $y(s) Q(0,y(s))$,
on which we have some grasp through their power series expansion, that
give information on the poles near $0$ and $\infty$ of $\Ftld$ and $\Gtld$.

Finally, we will often make use of the following proposition,
which allows us to go from meromorphic functions on $\C$ to
power series, mostly to obtain equations on
$Q(x,0)$ (resp. $Q(0,y)$) from equations on $\Ftld(s)$ (resp. $\Gtld(s)$).

\begin{proposition} \label{prop:lift_an_pow}
  Assume that a Laurent series $H(x) \in \C\llpar x))$ induces
  a meromorphic function at $x=0$. If $H(x(s)) = 0$
  or if $H(y(s)) = 0$ for $s$ near $0$, then $H(x) = 0$.
\end{proposition}
\begin{proof}
  Let $W$ be a neighborhood of $0$ such that $H(x(s)) = 0$ for all $s$ in $W \setminus \{0\}$.
  The function $x: \C \rightarrow \C$ is non-constant and analytic at $0$ with
  $x(0) = 0$. Hence by the open mapping theorem for holomorphic functions,
  the image of $W$ under $x$ is an open neighborhood of $0$. Thus, the
  analytic function $H(x)$ is locally zero at $0$, hence zero by
  analytic continuation. The argument is similar for $H(y(s)) = 0$.
\end{proof}

\section{Classification strategy}
\label{sec:class-strat}

In the previous section, we have constructed two
$q$-difference equations~(\ref{eq:f_qdiff})
and~(\ref{eq:g_qdiff}) satisfied by two meromorphic functions
on $\C$, whose differential properties reflect those of
the generating functions of quadrant walks $Q(x,y)$.
There are many results on the differential transcendence
of power series solution to a $q$-difference equation.
One of the first of those results was proved
by Ishizaki for the solutions of equations
of the form $y(qs) = a(s) y(s) + b(s)$~\cite{ishizaki98}.
We apply it to equations (\ref{eq:f_qdiff}) and~(\ref{eq:g_qdiff}):

\begin{proposition} \label{thm:ishizaki}
  The following statements are equivalent:
  \begin{enumerate}
  \item $\Ftld$ and $\Gtld$ are D-algebraic over $\C(s)$.
  \item $\Ftld$ and $\Gtld$ are in $\C(s)$.
  \end{enumerate}
\end{proposition}
\begin{proof}
  The real number $q$ is not a root of unit,
  the coefficients of (\ref{eq:f_qdiff})
  and (\ref{eq:g_qdiff})
  belong to $\C(s)$, and the functions $\Ftld$
  and $\Gtld$ are meromorphic at $s=0$. Hence we may apply
  Theorem~1.2 of~\cite{ishizaki98} to the two equations,
  which shows the claim.
\end{proof}

Thus, to investigate the D-algebraicity of $\Ftld$ (or equivalently $\Gtld$), the
strategy that we are going to explain in this section will consist
in determining for which weighted models these functions can
be rational. In the earlier paper applying this
strategy \cite{DreyfusHardouinRoquesSingerGenuszero}
(corresponding to the case without the interaction
statistics $a=b=1$), the authors prove that there exists no rational solution.
Whether it exists is highly dependent
on the coefficients of the $q$-difference equations considered.
For the classification of the models with the five supports of
Figure~\ref{fig:gen0_supports}, we find a general strategy
which allows us to handle almost all the cases uniformly,
its culmination being the classification in Theorem~\ref{thm:thm_clas}.

Because of the symmetries of the coefficients
of equations~(\ref{eq:f_qdiff}) and~(\ref{eq:g_qdiff}),
we will reduce the study of rational solutions
to these equations
to what we call \emph{decoupling equations}.
Recall that the kernel
curve $\Etproj$ admits a rational parametrization
$(x,y) : \P^1 \longrightarrow \Etproj$.
Consider
some fraction $h(x,y) \in \C(x,y)$.
The problem is to find two fractions $f(x)$ and $g(y)$
so that the following equation holds for
all points $s$ in $\P^1$,
\[
  h(x(s),y(s)) = f(x(s)) + g(y(s)).
\]
This is called an \emph{additive decoupling of $h$},
this notion being introduced in~\cite{BBMR16}.
Likewise, one can wonder if there exist
$f(x)$ and $g(y)$
so that for all points $s$ in $\P^1$,
\[
  h(x(s),y(s)) = f(x(s)) g(y(s)).
\]
This is called a \emph{multiplicative decoupling of $h$}
(introduced in~\cite{BMEFHR}).

The existence of such decouplings for a
fraction $h(x,y)$ plays an important role in the classification
of generating functions enumerating walks.
For instance, in the case of  quadrant walks with small steps without
interacting boundaries,
algebraicity is characterized by the finiteness of the group,
and the fact that the fraction $x y$
admits an additive decoupling (see
\cite{dreyfusEnumerationWeightedQuadrant2024}). In our case, the
equations that appear are generalizations of the decoupling
equations above, mixing the additive and multiplicative form, and they serve
the same purpose: we will see that the fact that they admit a solution
or not determines the position of $Q(x,y)$ in the differential
hierarchy. This explains our choice of terminology, as we introduce
them now.

\begin{lemma} \label{lem:rat_sol_dcpl}
  Assume that the functions
  $\Ftld(s)$ and $\Gtld(s)$ (defined in Section~\ref{sec:qdiff-equ}) are rational. In this
  case, define the following elements of $\C(s)$:
  \begin{align*}
    \ftld(s) \eqdef  \tfrac{1}{2} \left(\Ftld(s) + \Ftld^{\iota_1}(s)\right),
    & &
        \gtld(s) \eqdef  \tfrac{1}{2} \left(\Gtld(s)+\Gtld^{\iota_2}(s)\right), \\
    \ftld_h(s) \eqdef  \tfrac{1}{2} \left(\Ftld(s) - \Ftld^{\iota_1}(s)\right),
    & &
        \gtld_h(s) \eqdef  \tfrac{1}{2} \left(\Gtld(s) - \Gtld^{\iota_2}(s) \right).
  \end{align*}

  \begin{enumerate}
  \item The pair $(h_1(s),h_2(s))=(\ftld(s),\gtld(s))$ satisfies the inhomogeneous equation
    \begin{equation} \label{eq:eq_lin_inhom} \tag{$E_{\gamtld_1,\gamtld_2,\omega}$}
      \gamtld_1(s) h_1(s) + \gamtld_2(s) h_2(s) + \omega = 0 \text{ with $h_1^{\iota_1} = h_1$ and $h_2^{\iota_2} = h_2$}.
    \end{equation}
  \item The pair $(h_1(s),h_2(s)) = (\ftld_h(s),\gtld_h(s))$ satisfies the homogeneous equation
    \begin{equation} \label{eq:eq_lin_hom}\tag{$E'_{\gamtld_1,\gamtld_2}$}
      \gamtld_1(s) h_1(s) + \gamtld_2(s) h_2(s) = 0 \text{ with $h_1^{\iota_1} = - h_1$ and $h_2^{\iota_2} = - h_2$}.
    \end{equation}
  \end{enumerate}
  We refer to equations~(\ref{eq:eq_lin_inhom}) and (\ref{eq:eq_lin_hom}) as
  the \emph{decoupling equations}.
\end{lemma}
\begin{proof}
  Assuming that $\Ftld$ and $\Gtld$ belong to $\C(s)$,
  we first find a relation between
  $\Ftld^{\iota_1}$ and $\Gtld^{\iota_2}$
  (composition with $\iota_1$ and $\iota_2$ is always
  well defined for rational functions).
  Recall that $\Ftld(\tfrac{s}{q}) = \Ftld^{\iota_1 \iota_2}$, hence
  we may rewrite the $q$-difference equation (\ref{eq:f_qdiff}) into
  \begin{equation*}
    \left( \gamtld(s) \Ftld^{\iota_1}(s) + \frac{\omega}{\gamtld_2(s)} \right)^{\iota_2}
    = \gamtld(s) \Ftld(s) + \frac{\omega}{\gamtld_2(s)},
  \end{equation*}
  hence by applying $\iota_2$ on both sides we obtain
  \begin{equation} \label{eq:lem_rat_sol_dcpl_eq1}
    \gamtld(s) \Ftld^{\iota_1}(s) + \frac{\omega}{\gamtld_2(s)}
    = \left( \gamtld(s) \Ftld(s) + \frac{\omega}{\gamtld_2(s)} \right)^{\iota_2}.
  \end{equation}
  Moreover, the linear relation~(\ref{eq:lin_rel_ftld_gtld}) between $\Ftld$ and $\Gtld$
  can be rewritten as
  \begin{equation*}
    - \Gtld(s) =  \gamtld(s) \Ftld(s) + \frac{\omega}{\gamtld_2(s)},
  \end{equation*}
  so by applying $\iota_2$ we obtain
  \begin{equation} \label{eq:lem_rat_sol_dcpl_eq2}
    - \Gtld^{\iota_2}(s) = \left(\gamtld(s) \Ftld(s) + \frac{\omega}{\gamtld_2(s)}\right)^{\iota_2}.
  \end{equation}
  Eliminating the right-hand sides between (\ref{eq:lem_rat_sol_dcpl_eq1})
  and~(\ref{eq:lem_rat_sol_dcpl_eq2}),
  we extract the following relation between $\Ftld^{\iota_1}$ and $\Gtld^{\iota_2}$:
  \begin{align}
    \gamtld_1(s) \Ftld^{\iota_1}(s) + \gamtld_2(s) \Gtld^{\iota_2}(s) + \omega &= 0 \label{eq:eq_rwr_fi1}.
  \end{align}

  We copy for convenience Equation~(\ref{eq:lin_rel_ftld_gtld}):
  \begin{align}
    \gamtld_1(s) \Ftld(s) + \gamtld_2(s) \Gtld(s) + \omega &= 0\label{eq:eq_rwr_f}.
  \end{align}

  Taking the average of equations~(\ref{eq:eq_rwr_fi1}) and~(\ref{eq:eq_rwr_f}),
  one obtains (\ref{eq:eq_lin_inhom}).
  Taking half the difference of equations~(\ref{eq:eq_rwr_fi1}) and
  ~(\ref{eq:eq_rwr_f}),
  one obtains (\ref{eq:eq_lin_hom}).
  Indeed, $\iota_1$ is an involution, hence $(\Ftld^{\iota_1})^{\iota_1} = \Ftld$,
  so $\ftld^{\iota_1} = \ftld$ and $\ftld_h^{\iota_1} = - \ftld_h$. The same argument
  applies to~$\gtld$ and~$\gtld_h$.
\end{proof}

\begin{remark}
  From Proposition~\ref{prop:recap_gal},
  the condition $\ftld^{\iota_1} = \ftld$ is equivalent to the condition that
  there exists $f(x) \in \C(x)$ such that
  $\ftld(s) = f(x(s))$. Likewise, the condition $\ftld^{\iota_1} = - \ftld$
  asserts that there exists $\ftld(s)^2 = f(x(s))$ for some $f$,
  but that $\ftld(s)$ itself is not a function of $x(s)$. This explains
  the qualification of \emph{decoupling equations}: they relate functions
  in two different variables.
  \exqed
\end{remark}

We will see in the remaining of the section how the study
of the rational solutions of the
decoupling equations gives information on the series $Q(x,y)$,
either for showing its non-D-algebraicity in $x$ and $y$, or
for obtaining an explicit algebraic expression.

\subsection{Showing non D-algebraicity}

Our argument for showing non-D-algebraicity will rely on the fact that
we know that the solution $Q(x,y)$ is a
generating function of walks in the quadrant, and thus
we may control the expansion of $Q(x,0)$ and $Q(0,y)$ around $0$.

\begin{lemma} \label{lem:lift_laurent}
  Let $h$ be a fraction of $\C(s)$,
  and assume that the poles of $h$ belong to $\{0,\infty\}$.
  If $h^{\iota_1} = h$, then there exists a Laurent polynomial
  $H(x) \in \C[1/x]$ such that $H(x(s)) = h(s)$.
  Analogously, if $h^{\iota_2} = h$, then there
  exists a Laurent polynomial $H(y) \in \C[1/y]$ such that $H(y(s)) = h(s)$.
\end{lemma}
\begin{proof}
  Let $h \in \C(s)$ be a function whose poles belong to $\{0,\infty\}$,
  such that $h^{\iota_1} = h$.
  The extension $\C(s) / \C(x(s))$ being Galois
  of Galois group generated by $\iota_1$ by Proposition~\ref{prop:recap_gen0},
  there exists a fraction $H(x) \in \C(x)$ such that
  $H(x(s)) = h(s)$.
  Write $H(x) = \frac{U(x)}{V(x)}$
  with $U$ and $V$ relatively prime polynomials,
  $V$ monic, and let
  $u(s) \eqdef  U(x(s))$,
  $v(s) \eqdef  V(x(s))$.
  We write the polar divisor $(h)_{\infty}$ of $h$ in two different ways.

  First, we have by the assumption that $(h)_{\infty} = p \cdot 0 + q \cdot \infty$
  for some nonnegative integers $p$ and $q$. As $\iota_1(0) = \infty$ and
  $h^{\iota_1} = h$, we conclude
  that $p = q$, so by Proposition~\ref{prop:recap_gen0},
  \begin{equation} \label{eq:lem:lift_laurent_eq1}
    (h)_{\infty} = p \cdot 0 + p \cdot \infty = p \cdot (x(s))_0.
  \end{equation}

  Moreover, as $u$ and $v$ are polynomials in $x(s)$,
  we have that $(u)_{\infty} = \deg_x U(x)  \cdot (x(s))_{\infty}$.
  and $(v)_{\infty} = \deg_x V(x) \cdot (x(s))_{\infty}$.
  We also have that since the polynomials $U(x)$ and $V(x)$ are relatively
  prime, Bézout theorem implies the existence
  of some relation \[
    U'(x) U(x) + V'(x) V(x) = 1
  \]
  for $U'(x), V'(x) \in \C[x]$.
  Thus, by composing the relation with $x(s)$, we see that
  \[
    u'(s) u(s) + v'(s) v(s) = 1.
  \]
  For $s_0$ a pole of $x(s)$, $U(x(s_0))$ and $V(x(s_0))$
  are both nonzero since $U(x)$ and $V(x)$ are polynomials.
  For $s_0$ not a pole of $x(s)$, then $s_0$ is not a pole of $u', v', u, v$,
  and we see that it cannot be a zero of both $u(s)$ and $v(s)$.
  Thus, since
  \begin{align*}
    (h) &= (u)_0 + (v)_{\infty} - (u)_{\infty} - (v)_0 \\
        &= (u)_0 - (v)_0 + (\deg_x V(x) - \deg_x U(x)) \cdot (x(s))_{\infty}
  \end{align*}
  and the zeros of $u$ and $v$ don't compensate, we deduce that
  \begin{equation} \label{eq:lem:lift_laurent_eq2}
    (h)_{\infty} = (v)_0 - \min(\deg_x V(x) - \deg_x U(x), 0) \cdot (x(s))_{\infty}.
  \end{equation}

  Therefore, equating (\ref{eq:lem:lift_laurent_eq1}) and (\ref{eq:lem:lift_laurent_eq2}),
  we obtain the conditions
  \begin{multicols}{2}
    \begin{enumerate}[label=(\arabic*)]
    \item $\deg_x V(x) - \deg_x U(x) \ge 0$
    \item $(v)_0 = p \cdot (x(s))_{0}$.
    \end{enumerate}
  \end{multicols}
  Consider $w(s) \eqdef  \tfrac{v(s)}{x(s)^p}$. Using $(2)$, we compute
  its divisor as
  \begin{align*}
    (w) &= (v)_0 - (v)_{\infty} - p \cdot (x(s))_{0} + p \cdot (x(s))_{\infty} \\
        &= (p - \deg_x V(x)) \cdot (x(s))_{\infty}.
  \end{align*}
  Since $\deg\, (w) = 0$ (Proposition~\ref{prop:recap_divisors}),
  we deduce that $2 \cdot (p - \deg_x V(x)) = 0$,
  and thus $(w) = 0$, which implies that $w$ is a constant
  (Proposition~\ref{prop:recap_divisors}).
  As $V(x)$ is monic, this implies that
  $V(x) = x^p$.
  Moreover condition~$(1)$ implies that
  $\deg_x U(x) \le \deg_x V(x) = p$.
  We thus conclude that $H(x) = \tfrac{U(x)}{V(x)}$
  belongs to $\C[1/x]$.
  The proof for $h^{\iota_2} = h$ is similar.
\end{proof}

\begin{lemma} \label{lem:hypertrans}
  Assume that the following two conditions hold:
  \begin{enumerate}[label=(\arabic*)]
  \item \label{lem:hypertrans:item1} For any pair of solutions $(h_1,h_2)$ of \eqref{eq:eq_lin_inhom},
    one of the functions $h_1$ or $h_2$ has its poles in $\{0,\infty\}$.
  \item \label{lem:hypertrans:item2} The only solution of~\eqref{eq:eq_lin_hom} is $(0,0)$.
  \end{enumerate}
  Then $Q(x,y)$ is non $x$-D-algebraic nor $y$-D-algebraic.
\end{lemma}
\begin{proof}
  Assume that $Q(x,y)$ is $x$-D-algebraic or $y$-D-algebraic,
  then by Proposition~\ref{prop:Qx_Fs} the functions $\Ftld$ and $\Gtld$
  are rational. Hence, by Lemma~\ref{lem:rat_sol_dcpl}, the pair
  $(\ftld,\gtld)$ satisfies~(\ref{eq:eq_lin_inhom}) and the pair
  $(\ftld_h,\gtld_h)$ satisfies~(\ref{eq:eq_lin_hom}),
  with $\Ftld = \ftld + \ftld_h$ and $\Gtld = \gtld + \gtld_h$.

  From~\ref{lem:hypertrans:item1},
  assume without loss of generality that the poles of $\ftld$
  are in the set $\{0, \infty\}$.
  Then by Lemma~\ref{lem:lift_laurent} applied to $\ftld$,
  there exists a Laurent polynomial $f(x) \in \C[1/x]$
  such that $f(x(s)) = \ftld(s)$. Denote by $-d$ the
  valuation of $f(x)$, so that $P(x) \eqdef x^d f(x)$ is a
  polynomial of degree at most $d$.
  From~\ref{lem:hypertrans:item2}, we have that
  $\ftld_h = 0$, hence $\Ftld(s) = \ftld(s)$. Therefore, as $\Ftld(s)$
  is a continuation of $\Fcap(s)$, we have for $s \in V$ the
  equation
  \[
    x(s) Q(x(s),0) - f(x(s)) = 0.
  \]
  The function $x Q(x,0) - f(x)$ is meromorphic
  at $x=0$, hence by Lemma~\ref{prop:lift_an_pow} it is
  zero, so we have the equation \[
    x^{d+1} Q(x,0) = P(x).
  \]
  But we have that $Q(x,0) = 1 + O(x)$, which is a contradiction
  because $P(x)$ has degree at most~$d$.
\end{proof}

\subsection{Retrieving D-algebraic solutions}

In the remaining cases, we are able to prove that $Q(x,y)$ is D-algebraic
by lifting rational solutions (in $\C(s)$) to the decoupling equations~(\ref{eq:eq_lin_inhom})
and~(\ref{eq:eq_lin_hom}) to algebraic solutions (over $\C(x,y,t)$) of the
main functional equation (\ref{eq:eq_func_Qxy_general}).
Recall that $\omega = \frac{1}{ab} \left(a+b-a b\right)$, as defined in (\ref{notat:eq})

\begin{lemma} \label{lem:qdiff_rat}
  Assume that $\omega = 0$,
  and that $(h_1,h_2)$ is a nonzero solution to (\ref{eq:eq_lin_inhom})
  or (\ref{eq:eq_lin_hom}).
  Then the function $h_2$ satisfies the identity
  $\frac{h_2^{\sigma}(s)}{h_2(s)} = \frac{\gamtld^{\iota_1}(s)}{\gamtld(s)}$.
\end{lemma}
\begin{proof}
  Let $(h_1,h_2)$ be a pair solution
  to either (\ref{eq:eq_lin_inhom}) or (\ref{eq:eq_lin_hom}).
  In both cases, there exists some $\varepsilon \in \{-1,1\}$ such that
  $h_1^{\iota_1}(s) = \varepsilon h_1(s)$
  and $h_2^{\iota_2}(s) = \varepsilon h_2(s)$.

  Analogously to what was done
  in the first section, we start from the identity
  \begin{equation} \label{eq:lem_lift_insol1}
    \gamtld(s) h_1(s) = - h_2(s).
  \end{equation}
  Applying $\iota_1$ on both sides of the equation and using
  the relation $h_1^{\iota_1} = \varepsilon h_1$, we obtain
  \begin{equation} \label{eq:lem_lift_insol2}
    \varepsilon \gamtld^{\iota_1}(s) h_1(s) = - h_2^{\iota_1}(s).
  \end{equation}
  Eliminating $h_1(s)$ between~(\ref{eq:lem_lift_insol1}) and (\ref{eq:lem_lift_insol2}),
  and using the identity
  $h_2^{\iota_1}(s) = \varepsilon h_2^{\iota_2\, \iota_1}(s) = \varepsilon h_2^{\sigma}(s)$
  shows the claim.
\end{proof}

\begin{lemma} \label{lem:lift_sol}
  Assume that $\omega = 0$.
  \begin{enumerate}
  \item \label{lem:lift_sol:item1} If \eqref{eq:eq_lin_inhom} admits a nonzero
    solution $(h_1,h_2) \in \C(s)$, then $Q(x,y)$ is rational
    in $x$ and $y$
    (for the fixed $t$ of Proposition~\ref{prop:holom_eq}).
    More precisely, there exist
    $H_1(z), H_2(z) \in \C(z)$
    such that $H_1(x(s)) = h_1(s)$ and $H_2(y(s)) = h_2(s)$,
    and $\lambda \in \C$
    such that
    \begin{align*}
      x Q(x,0) &= \lambda H_1(x),
      & y Q(0,y) &= \lambda H_2(y).
    \end{align*}

  \item \label{lem:lift_sol:item2} If \eqref{eq:eq_lin_hom} admits a nonzero solution
    $(h_1,h_2) \in \C(s)$,
    then $Q(x,y)$ is algebraic over $\C(x,y)$ (for the fixed $t$ of
    Proposition~\ref{prop:holom_eq}). More precisely, there exist
    $H_1(z), H_2(z) \in \C(z)$ such that $H_1(x(s)) = h_1(s)^2$
    and $H_2(y(s)) = h_2(s)^2$, and $\lambda \in \C$ such that
    \begin{align*}
      x Q(x,0) &= \pm \lambda \sqrt{H_1(x)},
      & y Q(0,y) &= \pm \lambda \sqrt{H_2(y)}.
    \end{align*}
  \end{enumerate}
\end{lemma}
\begin{proof}
  Let $(h_1,h_2)$ be a nonzero
  solution to~(\ref{eq:eq_lin_inhom})
  or~(\ref{eq:eq_lin_hom}).
  By Lemma~\ref{lem:qdiff_rat}, the function $h_2(s)$ satisfies
  \begin{equation} \label{eq:lem_lift_qeq1}
    \frac{h_2^{\sigma}(s)}{h_2(s)} = \frac{\gamtld^{\iota_1}(s)}{\gamtld(s)}.
  \end{equation}
  Now, considering the function $H(s) \eqdef
  \frac{\Gtld(s)}{h_2(s)}$ (recall that $h_2$ is nonzero), we see by combining
  equations~(\ref{eq:lem_lift_qeq1}) and (\ref{eq:g_qdiff}) that
  $H(qs) = H(s)$. The function $H(s)$ is meromorphic on $\C$
  and $|q| \neq 1$, hence $H(s)$ is a constant.
  Therefore, there exists $\lambda \in \C$ such that
  $\Gtld(s) = \lambda h_2(s)$. By Equation~(\ref{eq:lin_rel_ftld_gtld}),
  we also deduce that $\Ftld(s) = \lambda h_1(s)$.

  We can now prove the two cases of the lemma:
  \begin{enumerate}
  \item[(\ref{lem:lift_sol:item1})] If $(h_1(s),h_2(s))$ is a solution to (\ref{eq:eq_lin_inhom}),
    then by Proposition~\ref{prop:recap_gal},
    there exist $H_1(x) \in \C(x)$ and $H_2(y) \in \C(y)$ such that
    $H_1(x(s)) = h_1(s)$
    and $H_2(y(s)) = h_2(s)$.
    Therefore, $x(s) Q(x(s),0) = \lambda H_1(x(s))$,
    and $y(s) Q(0,y(s)) = \lambda H_2(y(s))$ for all $s$ in $V$.
    Thus, as these functions are meromorphic at $s=0$,
    Proposition~\ref{prop:lift_an_pow} yields
    $x Q(x,0) = \lambda H_1(x)$ and $y Q(0,y) = \lambda H_2(y)$.
  \item[(\ref{lem:lift_sol:item2})] If $(h_1(s),h_2(s))$ is a solution to (\ref{eq:eq_lin_hom}),
    then $(h_1(s)^2)^{\iota_1} = h_1(s)^2$ and $(h_2(s)^2)^{\iota_2}
    = h_2(s)^2$. By Proposition~\ref{prop:recap_gal},
    there exist $H_1(x) \in \C(x)$ and
    $H_2(y) \in \C(y)$ such that $H_1(x(s)) = h_1(s)^2$
    and $H_2(y(s)) = h_2(s)^2$.
    Therefore, $x(s)^2 Q(x(s),0)^2 = \lambda^2 H_1(x(s))$
    and $y(s)^2 Q(0,y(s)^2 = \lambda^2 H_2(y(s))$ for all $s$ in $V$.
    Thus, as these functions are meromorphic at $s=0$,
    Proposition~\ref{prop:lift_an_pow} yields
    $x^2 Q(x,0)^2 = \lambda^2 H_1(x)$ and
    $y^2 Q(0,y)^2 = \lambda^2 H_2(y)$. \qedhere
  \end{enumerate}
\end{proof}

\begin{remark}
  \begin{enumerate}
  \item The value of $\lambda$ can be found by evaluating $H_1(x)$
    at $x=0$ (resp. $H_2(y)$ at $y=0$), and using the fact that
    $Q(0,0) = 1$.
  \item
    Note that the expressions for $H_1(x)$ and $H_2(y)$
    depend \emph{a priori} on the value of the real number $t$.
    However, in all the algebraic cases, the fractions $H_1(z)$
    and $H_2(z)$ will be fixed
    fractions of $\Q(d_{i,j}, a, b, t, z)$ analytic
    at $t=0$. Thus, we may lift the solutions for $Q(x,0)$ and $Q(0,y)$
    as formal power series in $t$.
  \end{enumerate}
  \exqed
\end{remark}

The classification will go as follows:
for every model and parameters, we will show that we are either
in the case of application of Lemma \ref{lem:hypertrans}
or Lemma \ref{lem:lift_sol}, by studying the decoupling equations (\ref{eq:eq_lin_inhom})
and (\ref{eq:eq_lin_hom}).

\section{Decoupling equations}  \label{sect:decoupl_eq}

In the previous section, we reduced the classification to the study
of two decoupling equations (\ref{eq:eq_lin_inhom}) and (\ref{eq:eq_lin_hom}).
According to the previously designed strategy, we will investigate
both equations separately.

\subsection{Homogeneous equation}
We first handle the homogeneous equation (\ref{eq:eq_lin_hom}),
which corresponds to the standard case of a \emph{multiplicative
  decoupling} (like for instance those in \cite{BMEFHR}),
in this case of the fraction $\gamtld \eqdef  \gamtld_1 / \gamtld_2$.
We want to determine in which cases this
equation admits rational solutions.
When they exist, we provide them explicitly.
Otherwise, we provide two distinct arguments
to show the non-existence of rational solutions.
The first one is standard, and revolves around a process
called \emph{pole propagation}.

\begin{lemma} \label{lem:propag_hom}
  Let $u(s)$ be in $\C(s)$.
  Assume that there exists a pole (resp. zero) $P \neq \{0, \infty\}$ of $u(s)$ such that
  for all $n$ in $\Z$ the point $\sigma^n P$ is never
  a zero (resp. pole) of $u(s)$.

  Then there is no nonzero $h(s) \in \C(s)$ such that
  $h^{\sigma}(s) = u(s) h(s)$.
\end{lemma}
\begin{proof}
  The result is elementary, see for
  instance~\cite[Lemma 3.5]{HS} for a proof. \qedhere
\end{proof}

We apply this lemma to the homogeneous equation (\ref{eq:eq_lin_hom}).
\begin{corollary} \label{cor:strange_poles}
  Assume that there exists $P \neq \{0, \infty\}$ a pole (resp. zero) of $\frac{\gamtld^{\iota_1}}{\gamtld}$
  such that for all $n$ in $\Z$ the point $\sigma^n P$ is never
  a zero (resp. pole) of $\frac{\gamtld^{\iota_1}}{\gamtld}$.
  Then the equation (\ref{eq:eq_lin_hom}) has no nonzero rational
  solution.
\end{corollary}
\begin{proof}
  Assume that $(h_1, h_2)$ is a nonzero solution to (\ref{eq:eq_lin_hom}).
  Then by Lemma~\ref{lem:qdiff_rat}, the function $h_2(s)$
  satisfies the equation $\frac{h_2^{\sigma}(s)}{h_2(s)} = \frac{\gamtld^{\iota_1}(s) }{ \gamtld(s)} $, which by assumption and Lemma~\ref{lem:propag_hom}
  has no nonzero rational solution, a contradiction.
\end{proof}

Unfortunately, this pole propagation technique does not discard
all cases where the homogeneous solution has no solution. Mainly,
it may happen that there exists a pair of fractions $(h_1,h_2)$
satisfying the linear relation with ${h'_1}^{\iota_1}(s) = \pm h'_1(s)$
and ${h'_2}^{\iota_2}(s) = \pm h'_2(s)$ (we call
such relaxed solutions \emph{signed solutions}).
In this case, Corollary~\ref{cor:strange_poles}
will not apply.
However, when such
solution with ``wrong'' signs for either ${h'_1}^{\iota_1}(s) / {h'_1}(s)$
or ${h'_2}^{\iota_2}(s) / h'_2(s)$ exists, we show
that (\ref{eq:eq_lin_hom}) has no nonzero rational solution
(with the ``right'' signs).

\begin{lemma} \label{lem:wrong_signs_hom}
  Let $(h'_1,h'_2)$ be a nonzero pair satisfying the relation
  $\gamtld_1(s) h'_1(s) + \gamtld_2(s) h'_2(s) = 0$
  with $(h'_1)^{\iota_1} = \pm h'_1$ and $(h'_2)^{\iota_2} = \pm h'_2$.
  If $(h'_1)^{\iota_1} = h'_1$ or $(h'_2)^{\iota_2} = h'_2$, then
  Equation~(\ref{eq:eq_lin_hom}) has no nontrivial rational
  solution.
\end{lemma}
\begin{proof}
  Assume that $(h'_1,h'_2)$ is such a pair, and let $(h_1, h_2)$
  be a pair of rational solutions to (\ref{eq:eq_lin_hom}).
  Then we have the equation
  \[
    \frac{h_1}{h'_1} = \frac{h_2}{h'_2} =: u.
  \]
  Now, by the symmetries of the $h_{1,2}$ and $h'_{1,2}$, we
  have that $(u^2)^{\iota_1} = (u^2)^{\iota_2} = u^2$. Therefore,
  the fraction $u^2$ is fixed by $\sigma$, hence $u^2 \in \C$,
  so $u \in \C$.

  Assuming that $(h'_1)^{\iota_1} = h'_1$, we obtain that
  $h_1^{\iota_1} = (u h'_1)^{\iota_1} = u h'_1 = h_1$.
  As $h_1$ satisfies also $h_1^{\iota_1} = - h_1$, we deduce
  that $h_1 = 0$, and thus that $h_2 = 0$.

  Similarly, assuming that $(h'_2)^{\iota_2} = h'_2$,
  we obtain that $(h_1,h_2) = (0,0)$.
\end{proof}

\subsection{Inhomogeneous equation}

To handle the inhomogeneous equation, we will proceed as in
Lemma~\ref{lem:propag_hom} by proving a propagation lemma
adapted to (\ref{eq:eq_lin_inhom}) (namely Lemma~\ref{lem:propag} below).
Unlike the previous case,
this lemma will only restrict the possible poles of
$h_1$ and $h_2$ to a finite set
of points of $\P^1$, for $(h_1,h_2)$ a solution.
In most cases, we will be able
to show that one of $h_{1,2}$ must have its poles
restricted to the set $\{0,\infty\}$,
which is one of the requirements of Lemma~\ref{lem:hypertrans}
which shows non-D-algebraicity of $Q(x,y)$ in $x$ and $y$.

Before we state and prove the propagation lemma for
the inhomogeneous equation, we recall two
easy facts on poles of rational maps on a curve.  The first fact
concerns the relations between the poles of two functions related by a
linear equation.

\begin{lemma} \label{lem:propag_null}
  Assume that $(h_1,h_2) \in \C(s)^2$ satisfies the relation
  $u_1 h_1 + u_2 h_2 + u_3 = 0$ for some $u_1$, $u_2$, $u_3$ in $\C(s)$.
  If $P$ is a pole of $h_1$ not in $\{ (u_1)_0, (u_2)_{\infty}, (u_3)_{\infty}\}$,
  then it is a pole of $h_2$.
\end{lemma}
\begin{proof}
  Assume that $P$ is a pole of $h_1(s)$. Since it is not a zero of $u_1(s)$, it is
  a pole of $u_1(s) h_1(s)$. By the relation, $P$ is a pole of $u_2(s) h_2(s) + u_3(s)$. Since $P$ is not
  a pole of $u_3(s)$, it is a pole of $u_2(s) h_2(s)$, and
  because it is not a pole of $u_2(s)$, this
  implies that $P$ is a pole of $h_2(s)$.
\end{proof}

The second standard fact concerns the poles of a function stable by automorphisms.
\begin{lemma} \label{lem:pole_aut_func}
  Let $h$ be in $\C(s)$, and $\tau$ an automorphism of $\P^1$.
  If $P$ is a pole (resp. zero) of $h$, then $\tau^{-1} P$ is a pole (resp. zero)
  of $h^{\tau}$.
  In particular, if $h^{\tau} = \lambda h$ for some nonzero $\lambda \in \C$,
  then the set of poles (resp. zeros) of $h$ is stable under the action of $\tau$.
\end{lemma}
\begin{proof}
  Assume that $P$ given by $s_0 \in \P^1$ is a zero of $h$.
  Then $h^{\tau}(\tau^{-1}(s_0)) = h(\tau(\tau^{-1}(s_0))) = h(s_0) = 0$.
\end{proof}

We now state the announced propagation lemma, specific to
(\ref{eq:eq_lin_inhom}).
It depends on the divisors of the coefficients of this
equation.
Recall from Proposition~\ref{prop:div_gamma} that
\begin{align} \label{eq:recall_div} \tag{$\star$}
  (\gamtld_1) &= P_1 + P_2 - 0 - \infty,
  & (\gamtld_2) &= P_3 + P_4 - 0 - \infty;
\end{align}
and that $\omega = 1 - A - B$ is a constant.
We define four finite sets $\cL_1^-, \cL_1^+, \cL_2^-, \cL_2^+ \subset \P^1$
as follows:
\begin{eqnarray} \label{eq:def_critic_points}
  \begin{aligned}
    \cL_1^- &\eqdef  \{ P_1, P_2, \iota_2 P_3, \iota_2 P_4 \},
    & \cL_1^+ &\eqdef  \iota_1 \cL_1^- = \{ \iota_1 P_1, \iota_1 P_2, \sigma^{-1} P_3, \sigma^{-1} P_4 \}, \\
    \cL_2^- &\eqdef  \{ \sigma P_1, \sigma P_2, \iota_2 P_3, \iota_2 P_4 \},
    & \cL_2^+ &\eqdef  \iota_2 \cL_2^+ = \{\iota_1 P_1, \iota_1 P_2, P_3, P_4\}.
  \end{aligned}
\end{eqnarray}
We call the elements of the sets $\cL_{1,2}^{+,-}$ the \emph{critical points}
of equation (\ref{eq:eq_lin_inhom}).

\begin{lemma}[Pole propagation] \label{lem:propag}
  Let $(h_1,h_2) \in \C(s)^2$ be a solution to (\ref{eq:eq_lin_inhom}).

  Let $P$ be a pole of $h_1$ distinct from $0$, $\infty$.
  \begin{enumerate}[label=(\roman*)]
  \item If $P \not\in \cL_1^-$, then $\sigma^{-1} P$ is a pole of $h_1$.
  \item If $P \not\in \cL_1^+$, then $\sigma P$ is a pole of $h_1$.
  \end{enumerate}

  Let $P$ be a pole of $h_2$ distinct from $0$, $\infty$.
  \begin{enumerate}[label=(\roman*')]
  \item If $P \not\in \cL_2^-$, then $\sigma^{-1} P$ is a pole of $h_2$.
  \item If $P \not\in \cL_2^+$, then $\sigma P$ is a pole of $h_2$.
  \end{enumerate}
\end{lemma}
\begin{proof}
  We prove (i). Let $P$ be a pole of $h_1$ distinct
  from $0$, $\infty$ and not in~$\cL_1^-$.
  Because $P \not\in \{P_1, P_2, 0, \infty\}
  = \{(\gamtld_1)_0, (\gamtld_2)_{\infty}, (\omega)_{\infty}\}$
  (see (\ref{eq:recall_div}) above), this implies by
  Lemma~\ref{lem:propag_null} that~$P$ is also a pole of~$h_2$.
  Now, $h_2^{\iota_2} =  h_2$, hence $\iota_2 P$ is a pole of $h_2$
  by Lemma~\ref{lem:pole_aut_func}.
  Now, $P \not\in \{ \iota_2 P_3, \iota_2 P_4, 0, \infty \}$
  hence $\iota_2 P \not\in \{P_3, P_4, 0, \infty\}
  = \{(\gamtld_2)_0, (\gamtld_1)_{\infty}, (\omega)_{\infty}\}$,
  so by Lemma~\ref{lem:propag_null}
  the point $\iota_2 P$ is a pole of $h_1$. Finally,
  $h_1^{\iota_1} =  h_1$, hence $\sigma^{-1} P = \iota_1 (\iota_2 P)$
  is a pole of $h_1$.

  We now prove (ii). Let $P$ be a pole of $h_1$
  distinct from $0$, $\infty$ and not in $\cL_1^+$.
  Since $P \not\in \cL_1^+ = \iota_1 \cL_1^-$, then
  $\iota_1 P \not\in \cL_1^-$, and $\iota_1 P \neq 0, \infty$. By (i), this implies that
  $\sigma^{-1} (\iota_1 P) = \iota_1 (\sigma P)$ is a pole of $h_1$.
  As $h_1^{\iota_1} =  h_1$, this implies that $\sigma P$ is a pole of $h_1$
  by Lemma~\ref{lem:pole_aut_func}.

  The proofs of $(i')$ and $(ii')$ are similar.
\end{proof}

Using the critical points, we may thus describe the possible poles of
rational solutions $(h_1,h_2)$ to (\ref{eq:eq_lin_inhom}).
\begin{lemma} \label{lem:sandwich_poles}
  Let $(h_1,h_2) \in \C(s)^2$ be a solution to (\ref{eq:eq_lin_inhom}).
  If $P$ is a pole of $h_1$ distinct from $0$ or $\infty$,
  then there exist two integers $m,n \ge 0$ such that $\sigma^{-m} P  \in \cL_1^-$
  and $\sigma^n P \in \cL_1^+$.
  Likewise, if $P$ is a pole of $h_2$ distinct from $0$ or $\infty$,
  then there exist two integers $m,n \ge 0$ such that $\sigma^{-m} P \in \cL_2^-$
  and $\sigma^n P \in \cL_2^+$.
\end{lemma}
\begin{proof}
  Assume for the sake of contradiction that $P \neq 0, \infty$ is a pole of
  $h_1$ satisfying $\sigma^n P \not\in \cL_1^+$ for all $n \ge
  0$. Then by induction and Lemma~\ref{lem:propag},
  we show that $\sigma^n P$ is a pole of $h_1$ for all $n \ge 0$.
  Since $P$ is distinct from $0$ or
  $\infty$, the orbit $(\sigma^n P)_{n \ge 0}$ is infinite, hence the
  fraction $h_1$ has an infinite number of poles, a contradiction.
  The other points are proved in a similar fashion.
\end{proof}

\subsection{\texorpdfstring{$\boldsymbol{\sigma}$-distance}{σ-distance}} \label{sec:sigma-dist}

Thanks to pole propagation, Lemma~\ref{lem:sandwich_poles} allows us
to locate the possible poles of $h_1$ and $h_2$
for a rational solution $(h_1,h_2)$ to (\ref{eq:eq_lin_inhom}).
Similarly, Lemma~\ref{lem:propag_hom} gives a sufficient
condition for proving that (\ref{eq:eq_lin_hom}) has
no nonzero rational solution $(h'_1,h'_2)$. These two
lemmas thus give conditions to the existence of solutions
to decoupling equations based on the relations between the points
of the finite sets $\cL_{1,2}^\pm$.
These relations are captured by a signed distance
called the \emph{$\sigma$-distance}, that we introduce
to compare two points of $\P^1$ with respect
to the action of the group $<\sigma>$.
This will give us numerical
data from which we will conduct the classification.

\begin{defprop} \label{def:difference}
  Let $P$ and $P'$ be two points of $\P^1$ distinct from $0$ and $\infty$.
  We define the \emph{$\sigma$-distance} $\delta(P,P')$
  of the points $P$ and $P'$ as follows:
  \begin{itemize}
  \item If there exists an integer $n \in \Z$ such that $\sigma^n P = P'$,
    then $n$ is unique, and we define $\delta(P,P') = n$.
  \item Otherwise, if no such integer exists, we define $\delta(P,P') = \bot$.
  \end{itemize}
\end{defprop}
\begin{proof}
  We just need to show the uniqueness. Assume that $\sigma^n P = \sigma^m P$
  for two integers~$m$ and~$n$.
  Then $\sigma^{n-m} P = P$, which is possible if and only if
  if $n=m$ because $0$ and $\infty$ are the only periodic points of the action of $\sigma$
  on $\P^1$ (indeed, if $q^ns = s$ with $n\ge 1$, then $s=0$ or
  $s=\infty$, for $q^n \neq 1$).
\end{proof}

The $\sigma$-distance satisfies the standard arithmetic properties
that one would expect from a signed distance.
\begin{proposition} \label{prop:comp_diff}
  With the convention that $n + \bot = \bot$ for every integer $n$,
  and that $\bot = - \bot$, the $\sigma$-distance satisfies the following
  properties for points $P$, $P'$ and $P''$ distinct from $0$ or $\infty$:
  \begin{enumerate}[label=(\roman*)]
  \item $\delta(P,P') = -\delta(P',P)$,
  \item $\delta(P,P')+\delta(P',P'') = \delta(P,P'')$ if $\delta(P,P')$ and $\delta(P',P'')$
    are finite,
  \item $\delta(P,\sigma(P')) = \delta(P,P') + 1$,
  \item $\delta(P,P') = \delta(\iota_1 P', \iota_1 P) =  \delta(\iota_2 P',\iota_2 P)$.
  \end{enumerate}
\end{proposition}
\begin{proof}
  The proofs of (i), (ii) and (iii) are straightforward, hence
  we focus on the proof of (iv).

  Assume first that $\sigma^n P = P'$ for some integer $n$.
  Recall that $\sigma = \iota_2 \iota_1$ with $\iota_1^2 = \iota_2^2 = \id$,
  so it is easy to see that $\iota_1 \sigma^n = \sigma^{-n} \iota_1$.
  Thus, $\sigma^{-n} (\iota_1 P) = \iota_1 (\sigma^n P)
  = \iota_1 P'$, which implies that $\delta(\iota_1 P, \iota_1 P')$
  is finite, being equal to $-n = - \delta(P, P') = \delta(P',P)$.
  The application $\iota_1$ is an involution, thus if
  $\delta(P,P') = \bot$, then we also have $\delta(\iota_1 P,\iota_1 P') = \bot$.
\end{proof}

We are going to determine the values of~$\delta(P,P')$ for all~$(P,P') \in \cL_1^-
\times \cL_1^+$ and $(P,P') \in \cL_2^- \times \cL_2^+$. These values
are compiled respectively in the matrices~$M_1$ and~$M_2$
in~$\mathcal{M}_4(\Z \cup \{\bot\})$, the lines~($\cL_{1,2}^-$) and
columns~($\cL_{1,2}^+$) being ordered as in~(\ref{eq:def_critic_points}).
More precisely, their entries are organized as follows:
\begin{eqnarray} \label{mat:matrices}
  \begin{aligned}
    M_1 &\eqdef \begin{pmatrix}
      \delta(P_1,\iota_1 P_1) & \delta(P_1,\iota_1 P_2) & \delta(P_1,\sigma^{-1} P_3) &
                                                                                        \delta(P_1,\sigma^{-1} P_4)\\
      \delta(P_2,\iota_1 P_1) & \delta(P_2,\iota_1 P_2) & \delta(P_2,\sigma^{-1} P_3) &
                                                                                        \delta(P_2,\sigma^{-1} P_4)\\
      \delta(\iota_2 P_3,\iota_1 P_1) & \delta(\iota_2 P_3,\iota_1 P_2) & \delta(\iota_2 P_3,\sigma^{-1} P_3) &
                                                                                                                \delta(\iota_2 P_3,\sigma^{-1} P_4)\\
      \delta(\iota_2 P_4,\iota_1 P_1) & \delta(\iota_2 P_4,\iota_1 P_2) & \delta(\iota_2 P_4,\sigma^{-1} P_3) &
                                                                                                                \delta(\iota_2 P_4,\sigma^{-1} P_4)\\
    \end{pmatrix}, \\
    M_2 &\eqdef \begin{pmatrix}
      \delta(\sigma P_1,\iota_1 P_1) & \delta(\sigma P_1,\iota_1 P_2) & \delta(\sigma P_1,P_3) &
                                                                                                 \delta(\sigma P_1,P_4)\\
      \delta(\sigma P_2,\iota_1 P_1) & \delta(\sigma P_2,\iota_1 P_2) & \delta(\sigma P_2,P_3) &
                                                                                                 \delta(\sigma P_2,P_4)\\
      \delta(\iota_2 P_3,\iota_1 P_1) & \delta(\iota_2 P_3,\iota_1 P_2) & \delta(\iota_2 P_3,P_3) &
                                                                                                    \delta(\iota_2 P_3,P_4)\\
      \delta(\iota_2 P_4,\iota_1 P_1) & \delta(\iota_2 P_4,\iota_1 P_2) & \delta(\iota_2 P_4,P_3) &
                                                                                                    \delta(\iota_2 P_4,P_4)\\
    \end{pmatrix}.
  \end{aligned}
\end{eqnarray}

\begin{proposition} \label{prop:mat_sym}
  The matrices $M_1$ and $M_2$ satisfy the following relations:
  \begin{enumerate}[label=(\roman*)]
  \item $M_1^T = M_1$ and $M_2^T = M_2$,
  \item $M_2 = M_1 + \begin{pmatrix} -J_2 & 0
    \\ 0 & J_2 \end{pmatrix}$
  where $J_2 = \begin{pmatrix} 1 & 1 \\ 1 & 1 \end{pmatrix}$.
\end{enumerate}
\end{proposition}
\begin{proof}
  These are straightforward applications of Proposition~\ref{prop:comp_diff}.
\end{proof}
Thanks to the above proposition, it is only required to
compute the $\sigma$-distances of $10=4+3+2+1$ pairs of points, namely those above the
diagonal of $M_1$. Note that the sets $\cL_{1,2}^{\pm}$,
and thus the matrix $M_1$ depend on the set of steps $\calS_i$ of
the model and weights $d_{i,j}$, $A$ and $B$.
We finish this subsection by giving two lemmas to exploit
these matrices.

\begin{lemma} \label{lem:noseg_nopol}
  Let $(h_1,h_2)$ be a pair of rational solutions to (\ref{eq:eq_lin_inhom}).
  \begin{enumerate}[label=(\roman*)]
  \item If all the entries of $M_1$ are in $\Z^{-} \cup \{\bot\}$, then the
    poles of $h_1(s)$ belong to $\{0, \infty\}$.
  \item If all the entries of $M_2$ are in $\Z^{-} \cup \{\bot\}$, then the
    poles of $h_2(s)$ belong to $\{0, \infty\}$.
  \end{enumerate}
\end{lemma}
\begin{proof}
  Let $(h_1,h_2)$ be a pair of rational solutions to (\ref{eq:eq_lin_inhom}).
  We prove $(i)$. Assume that all the entries of~$M_1$ are in $\Z^{-} \cup \{\bot\}$.
  If $P \neq \{0, \infty\}$ is a pole of $h_1$, then by Lemma~\ref{lem:sandwich_poles},
  there exist $m, n \ge 0$ such that $\sigma^{-m} P =: Q^- \in \cL_1^-$
  and $\sigma^n P =: Q^+\in \cL_1^+$. But then by
  (ii) of Proposition~\ref{prop:comp_diff}, one
  has $\delta(Q^-,Q^+) = \delta(Q^-,P) + \delta(P,Q^+) = n+m \ge 0$,
  a contradiction since $\delta(Q^-,Q^+)$ is an entry of $M_1$. Thus,
  $h_1$ has no poles besides $0$ and $\infty$.
  The proof of point $(ii)$ is similar.
\end{proof}

\begin{lemma} \label{lem:norel}
  If one of the rows of $M_1$ consists of $\bot$'s only, then
  (\ref{eq:eq_lin_hom}) has no nonzero rational
  solution.
\end{lemma}
\begin{proof}
  From Proposition~\ref{prop:div_gamma},
  we may write the divisor of $\gamtld^{\iota_1} / \gamtld$
  as \[
    (\gamtld^{\iota_1} / \gamtld) = \iota_1 P_1 + \iota_1 P_2 + P_3 + P_4
    - P_1 - P_2 - \iota_1 P_3 - \iota_1 P_4.
  \]
  Assume that there exists $Q^- \in \cL_1^-$ such that
  for all $Q^+ \in \cL_1^+$ one has $\delta(Q^-,Q^+) = \bot$
  ($Q^-$ labels the row $M_1$ consisting of $\bot$'s only).
  Then there exists an integer $k$ such
  that $Q' \eqdef  \sigma^{k} Q^-$
  is a pole of $\gamtld^{\iota_1} / \gamtld$,
  namely $k=0$ for $Q^- \in \{P_1, P_2\}$ or
  $k = -1$ for $Q^- \in \{\iota_2 P_3, \iota_2 P_4\}$.
  \begin{itemize}
  \item The point $Q'$ is a pole of $\gamtld^{\iota_1} / \gamtld$.
  \item The point $\sigma^n Q'$ is never a zero of $\gamtld^{\iota_1} / \gamtld$.
    Indeed, if it were the case, then $Q^+\eqdef \sigma^m Q'$ would belong to $\cL_1^+$,
    either for $m=n$ if $\sigma^n Q' \in \{\iota_1 P_1, \iota_1 P_2\}$,
    or $m=n-1$ if $\sigma^n Q' \in \{ P_3, P_4\}$.
    The point (iii) of Proposition~\ref{prop:comp_diff} would then
    imply that \[
      \delta(Q^-,Q^+) = \delta(Q^-,Q') + \delta(Q',Q^+) = k + m,
    \]
    while the row of $M_1$ corresponding to $Q^-$ consisting of $\bot$'s only
    implies that $\delta(Q^-,Q^+) = \bot$, a contradiction.
  \end{itemize}
  Therefore by Corollary~\ref{cor:strange_poles}, (\ref{eq:eq_lin_hom})
  has no nonzero rational solution.
\end{proof}

\section{Computing the $\sigma$-distance} \label{sect:decide_diff}

Denote by $\Fpar$ the field $\Q(d_{i,j},a,b)$.
In this section, we describe a heuristic to decide given two points $P$ and $P'$ of
$\P^1$ if there exists an integer~$n$ such that $\sigma^n(P) = P'$.
In other words, for two points $P$ and $P'$ of $\P^1$,
the goal is to compute the $\sigma$-distance
$\delta(P,P')$ of Definition~\ref{def:difference}.
When we restrict to the points
that originate from zeros of the
fractions $\gamtld_1$ and $\gamtld_2$, the
$\sigma$-distance is computable. The main argument used here,
and already exploited for instance in \cite{BMM}, is based on \emph{valuations}.

\begin{defprop}
  There exists an embedding (a $\Fpar$-algebra homomorphism)
  $\psi : \overline{\Fpar(t)} \longrightarrow \C^{frac}\llpar T \rrpar$
  where $\C^{frac} \llpar T \rrpar$ is the field of formal Puiseux series
  over $\C$ in the formal variable $T$.
  \footnote{Although the real number $t$ is transcendental over $\Fpar$,
    we do not directly consider the field of formal Puiseux series in $t$
    to avoid conflicts of notation with the usual sum of complex numbers.}
  We fix once and for all this embedding $\psi$.
  As a result, for any $u \in \overline{\Fpar(t)}$, we define its
  \emph{valuation} $v(u) \in \Q$ to be the valuation in the variable $T$ of~$\psi(u)$
  (the valuation of $0$ being defined as $+\infty$).
\end{defprop}
\begin{proof}
  Consider the embedding $\psi_5 : \Fpar(t) \longrightarrow \C^{frac} \llpar T \rrpar$
  defined as the composition
  \[\begin{tikzcd}
      {\Fpar(t)} & {\Fpar(T)} & {\C(T)} & \C(\!(T)\!) & {\C^{frac}\llpar T \rrpar}
      \arrow["\psi_1", from=1-1, to=1-2]
      \arrow["\psi_2", from=1-2, to=1-3]
      \arrow["\psi_3", from=1-3, to=1-4]
      \arrow["\psi_4", from=1-4, to=1-5]
    \end{tikzcd}\]
  where the embedding $\psi_1 : t \mapsto T$ is the isomorphism between
  $\Fpar(t)$ and $\Fpar(T)$ as $t$ is transcendental over $\Fpar$;
  the embedding $\psi_2$ is the map induced by the inclusion $\Fpar \subset \C$;
  the embedding $\psi_3$ is the map from $\C(T)$ into
  the field of Laurent series $\C(\!(T)\!)$; and the embedding $\psi_4$
  is an arbitrary embedding of $\C(\!(T)\!)$ into
  $\C^{frac} \llpar T \rrpar$ its algebraic closure (this follows from the Newton-Puiseux
  theorem since $\C$ is algebraically closed of characteristic zero
  \cite[Th. 6.1.5]{stanleyEnumerativeCombinatoricsVolume2023}).
  As the field $\C^{frac} \llpar T \rrpar$ is algebraically closed,
  the embedding $\psi_5$ admits an extension to
  an embedding $\psi : \overline{\Fpar(t)} \longrightarrow \C^{frac} \llpar T \rrpar$
  (\cite[Th. V.2.8]{Lang}).
\end{proof}

\begin{definition}
  Let $P \in \P^1 \setminus \{0,\infty\}$ such
  that $\phi(P) = ([1:x_1],[1:y_1]) \in \P^1 \times \P^1$
  with $x_1, y_1 \in \overline{\Fpar(t)}$. Then define
  the \emph{bivaluation} of $P$ to be $v(P) \eqdef  (v(x_1),v(y_1))$.
\end{definition}

\begin{lemma} \label{prop:alg_coord_poles}
  Let $H(x,y) \in \Fpar(t)(x,y)$ be a fraction such that
  $h(s) \eqdef  H(x(s),y(s)) \in \C(s)$ is well defined.
  If $P \in \P^1 \setminus \{0, \infty\}$
  is a pole or zero of $h$,
  then the point $\phi(P) = ([1:x_1], [1:y_1]) \in \Etproj \subset \P^1 \times \P^1$
  is distinct from $\Omega = (0,0)$,
  and $x_1$ and $y_1$ are algebraic over~$\Fpar(t)$.
\end{lemma}
\begin{proof}
  Let $s_0$ be the coordinate in $\P^1$ of $P$. By assumption,
  $s_0 \neq 0, \infty$, hence Proposition~\ref{prop:recap_gen0} yields
  $x(s_0) \neq 0$ and $y(s_0) \neq 0$.
  Moreover, the functions $x(s)$ and $y(s)$ belong to $\overline{\Fpar(t)}(s)$. Hence,
  $h(s) \in \overline{\Fpar(t)}(s)$, thus if $s_0$ is a pole or zero
  of $h(s)$, then $s_0$ belongs to $\overline{\Fpar(t)}$. Therefore,
  so do $x_1 = \tfrac{1}{x(s_0)}$ and $y_1 = \tfrac{1}{y(s_0)}$.
\end{proof}

The above lemma allows us to talk about the bivaluations
of the zeros of $\gamtld_1$ and $\gamtld_2$
(the points $P_i$ defined in Proposition~\ref{prop:div_gamma}).
Now, recall the expressions for $\iota_1$ and $\iota_2$
on $\Etproj \subset \P^1 \times \P^1$ of (\ref{eq:def_i1}) and (\ref{eq:def_i2}):
\begin{eqnarray} \label{eq:group_repeat}
  \begin{aligned}
    \iota_1([1:x_1],[1:y_1]) & = \left([1:x_1],\left[1:\frac{d_{-1,1} x_1^2+d_{0,1} x_1 + d_{1,1}}{d_{1,-1} y_1}\right]\right), \\
    \iota_2([1:x_1],[1:y_1]) & = \left(\left[1:\frac{d_{1,-1} y_1^2+d_{1,0}y_1 + d_{1,1}}{d_{-1,1} x_1}\right],[1:y_1]\right).
  \end{aligned}
\end{eqnarray}

Hence, for a point $P \in \P^1 \setminus \{0, \infty\}$,
the homogeneous coordinates at infinity of
$\phi(\iota_1 P) = \iota_1 (\phi(P))$
and $\phi(\iota_2 P ) = \iota_2 (\phi(P))$
(Proposition~\ref{prop:recap_gen0})
are explicit rational functions in the coordinates of $\phi(P)$.
Hence, if the
coordinates $x_1$ and $y_1$ of $\phi(P)$ are in $\overline{\Fpar(t)}$,
then so are the coordinates of $\phi(\iota_1 P)$
and  $\phi(\iota_2 P)$. Thus, all
the points of the sets $\cL_{1,2}^{\pm}$ defined
in (\ref{eq:def_critic_points}) admit a bivaluation.
This raises
the possibility of keeping track of the successive bivaluations
of the points~$\sigma^n(P)$ for all integers $n$
and $P \in \cL_{1,2}^{\pm}$. It turns
out that in most cases, the bivaluations of $\sigma(P)$ and $\sigma^{-1}(P)$
depend only on the bivaluation of~$P$.

\begin{lemma} \label{thm:track_val}
  Let $P$ be in $\P^1$ with $\phi(P) = ([1:x_1],[1:y_1])$ and $x_1,y_1
  \in \overline{\Fpar(t)}$, and let $(i,j)$ be
  $v(P)$ the bivaluation of $P$.
  \begin{enumerate}[label=(\arabic*)]
  \item If $i < 0$ then $v(\iota_1(P)) = (i, 2 i - j)$.
  \item If $j < 0$ then $v(\iota_2(P)) = (2 j - i, j)$.
  \end{enumerate}
\end{lemma}
\begin{proof}
  We show (1). Write $\iota_1([1:x_1],[1:y_1]) = ([1:x_1],[1:y'_1])$
  (as read in (\ref{eq:group_repeat})).
  If $v(x_1) < 0$, then since $d_{-1,1} \neq 0$,
  we have $v(d_{-1,1} x_1^2) = 2 v(x_1) = 2 i$ and
  $v(d_{0,1} x_1 + d_{1,1}) \ge v(x_1) > 2 v(x_1) = 2 i$.
  Thus, the numerator of $y'_1$ has valuation $2 i$.
  Moreover, since $d_{1,-1} \neq 0$,
  we compute the valuation of the denominator of $y'_1$
  as $v(d_{1,-1} y_1) = v(y_1) = j$, hence the result.
  The proof of (2) is similar.
\end{proof}

\begin{lemma} \label{lem:dyn_sys}
  Let $P \in \P^1$ with $\phi(P) = ([1:x_1],[1:y_1])$ and $x_1, y_1 \in \overline{\Fpar(t)}$,
  and let $v(P) = (i, j)$ and $\delta = |i-j|$.
  \begin{enumerate}[label=(\arabic*)]
  \item If $i < j < 0$ then
    $v(\sigma^k(P)) = (i - 2 \delta k, j - 2 \delta k)$
    for all $k \ge 0$,
  \item If $j < i < 0$ then
    $v(\sigma^{-k}(P)) = (i - 2 \delta k, j - 2 \delta k)$
    for all $k \ge 0$.
  \end{enumerate}
\end{lemma}
\begin{proof}
  We prove (1).
  Assume that $\phi(P) = ([1:x_1],[1:y_1])$
  with $x_1, y_1 \in \overline{\Fpar(t)}$
  and $(i,j) \eqdef  v(P)< 0$.
  We first compute the bivaluation of $\iota_1 P$.
  As $i < 0$, Lemma~\ref{thm:track_val} asserts
  that $v(\iota_1(P))$ is completely determined by $i$ and $j$. Thus,
  \[
    v(\iota_1 P) = (i, 2i -j) = (i, j + 2 (i-j)) = (i, j - 2 \delta) \mbox{ since $i < j$}.
  \]
  As $j - 2 \delta < 0$, the bivaluation of $\sigma(P) = \iota_2 (\iota_1 P)$ is also
  completely determined by $i$ and $j$
  by Lemma~\ref{thm:track_val}, and thus
  \[
    v(\sigma P) = (2 (j - 2 \delta) - i, j - 2 \delta) = (3i - 2j, j - 2 \delta) = (i - 2 \delta, j - 2 \delta).
  \]
  We thus proved that if $v(P) = (i,j)$ with $i < j < 0$,
  then $v(\sigma P) = (i - 2 \delta, j - 2 \delta)$.
  An easy induction completes the proof.
  The proof of the second point is similar.
\end{proof}

As the classification of the nature of $Q(x,y)$ given a weighted
model depends on the matrices $M_1$ and $M_2$ defined
in (\ref{mat:matrices}), we need to compute
$\delta(Q^-,Q^+)$ for
two points $Q^- \in \cL_1^-$ and $Q^+ \in \cL_1^+$
for these weights $d_{i,j}$, $a$, $b$.
Thus, we are able to determine the tables of Appendix~\ref{sect:mat_results}.
Note that from Lemma~\ref{prop:alg_coord_poles},
both points $Q^-$ and $Q^+$, with elements of their orbit
under $\sigma$ have a bivaluation. We follow the following method.
\begin{enumerate}
\item Compute $\sigma^n Q^-$
  for $n \in \{-2,-1,0,1,2\}$. It happens in all cases
  (when $Q^- \in \cL_1^-$)
  that $v(\sigma^{-2}(Q^-)) = (i,j)$ with $j < i < 0$
  and $v(\sigma^{2}(Q^-)) = (i,j)$ with $i < j < 0$.
  Thus, Lemma~\ref{lem:dyn_sys}
  allows us to determine the sequence of bivaluations
  of $\sigma^n(Q^-)$ for all $n \in \Z$.
\item Determine the bivaluation $(i', j')$ of the point
  $Q^+$.
  \begin{itemize}
  \item If one point of the orbit of $Q^-$ has bivaluation $(i', j')$
    (which we may decide, see the above point),
    then there are a finite number of $n$ such that
    $v(\sigma^n(Q^-)) = (i', j')$. For each of these $n$,
    check if $\sigma^n(Q^-) = Q^+$. If one of these $n$ works, then
    $\delta(Q^-,Q^+) = n$.
  \item Otherwise, if the bivaluation $(i',j')$ does not appear
    in the orbit of $Q^-$ or no $n$ works, then  $\delta(Q^-,Q^+) = \bot$.
  \end{itemize}
\end{enumerate}

Of course, the space of parameters $a$, $b$, $d_{i,j}$ is infinite. Hence,
the actual procedure adds the following level of complexity: the bivaluation
of a point $Q$ depends on an algebraic condition on the parameters. Thus,
we need to explore all the possible bivaluations according to these parameters
(for the points $Q^-$ and $Q^+$).
The core of the procedure stays the same.
Instead of giving a dry algorithm, we expand below an example.

\begin{example}
  We consider the set of steps $\calS_1$ of Figure~\ref{fig:gen0_supports}.
  In this example, we show how to construct
  the entry $\delta(P_2,\iota_1 P_2)$ of Table~\ref{tab:mat_results1},
  depending on the weights
  $d_{0,1}$, $d_{1,-1}$, $d_{-1,1}$, $A$ and~$B$.

  \begin{enumerate}
  \item The first step consists in computing the bivaluation of $P_2$
    depending on the weights $d_{i,j}$, $A$ and $B$.  We first compute the
    coordinates of $\phi(P_2)$ for generic $d_{i,j}$, $A$ and $B$:
    \[ \phi(P_2) = \left(\left[1: \frac{-d_{0,1} d_{1,-1}t^2}{d_{1,-1}
            d_{-1,1} t^2 + A^2 - A}\right], \left[1:\frac{-A d_{0,1} t}{d_{1,-1}
            d_{-1,1} t^2 + A^2 - A}\right]\right).
    \] The Laurent series expansions of the coordinates of $P_2$ at
    $t=0$ are as follows \[ \phi(P_2) = \left(\left[1:\frac{-d_{0,1}
            d_{1,-1} t^2}{A(A-1)} + O(t^3)\right], \left[1:\frac{-A d_{0,1}
            t}{A(A-1)} + O(t^2)\right]\right).
    \]

    We thus notice that when $A \neq 0$, the bivaluation of $P_2$
    is $v(P_2) = (2,1)$ (the weights $d_{i,j}$ are nonzero,
    and note that as $A=1-\tfrac{1}{a}$, then $A$ cannot be equal to $1$).
    Otherwise, in the case $A=0$, we find $v(P_2) = (0,\infty)$.
    We now compute their orbits in these two cases.

    \begin{enumerate}
    \item Assume that $A \neq 0$, so that $v(P_2) = (2,1)$. We check by
      computing the points $(\sigma^i P_2)_{-2 \leq i \leq 2}$ that their
      bivaluations do not depend on $A$ as long at it
      is nonzero, nor the weights
      $d_{1,-1}$, $d_{-1,1}$ and $d_{0,1}$,
      and that they are equal to
      \begin{equation} \label{eq:21}
        \dots \rightarrow_{\sigma} (-2,-3) \rightarrow_{\sigma} (0,-1)
        \rightarrow_{\sigma} v(P_2) = (2,1) \rightarrow_{\sigma} (0,1)
        \rightarrow_{\sigma} (-2,-1) \rightarrow_{\sigma} \dots.
      \end{equation}
      The remaining parts of the above sequence may be continued
      using Lemma~\ref{lem:dyn_sys}.
      Indeed, $v(\sigma^{-2} P_2) = (-2,-3)$, hence $\sigma^{-2} P_2$
      satisfies condition (2) of
      Lemma~\ref{lem:dyn_sys}, hence we know that whatever the value of $A \neq 0$,
      one has $v(\sigma^{-k-2} P_2) = (-2 - 2k, -3 - 2k)$. Similarly,
      $v(\sigma^2 P_2) = (-2,-1)$ satisfies condition (1) of Lemma~\ref{lem:dyn_sys},
      hence we deduce that $v(\sigma^{k+2} P_2) = (-2+2k, -1+2k)$
      regardless of the values of the weights $d_{i,j}$ and $A \neq 0$.

    \item For $A=0$, using the same technique, we compute the
      sequence of bivaluations for $(\sigma^i P_2)_{-3 \le i \le 1}$ and $A=0$
      (with $v(P_2) = (0, \infty)$):
      \begin{equation} \label{eq:22}
        \dots \rightarrow_{\sigma} (-2,-3) \rightarrow_{\sigma} (0,-1)
        \rightarrow_{\sigma} (\infty,\infty) \rightarrow_{\sigma} v(P_2) = (0,\infty)
        \rightarrow_{\sigma} (-2,-1) \rightarrow_{\sigma} \dots,
      \end{equation}
      Again, the remaining of the above sequence
      may be continued using Lemma~\ref{lem:dyn_sys}.
    \end{enumerate}

  \item We now compute $\delta(P_2, \iota_1 P_2)$, for $A=0$
    and $A \neq 0$.
    We first compute the coordinates of
    $\phi(\iota_1 P_2)$ in $\P^1 \times \P^1$
    for generic values of the weights $d_{i,j}$, $A$ and $B$:
    \[\phi(\iota_1 P_2) = \left(\left[1:-\frac{d_{0,1} d_{1,-1}
            t^2}{d_{1,-1} d_{-1,1} t^2 + A^2 - A}\right], \left[1:\frac{d_{0,1} t
            (A - 1)}{d_{1,-1} d_{-1,1} t^2 + A^2 - A} \right]\right). \] We find
    that
    \[ v(\iota_1 P_2) = \left\{\begin{array}{lr} (2,1)
      & \mbox{ if $A \neq 0$} \\ (0, -1) & \mbox{if $A = 0$} \end{array}\right. .\]

\begin{enumerate}
\item For $A \neq 0$, since $v(\iota_1 P_2) = (2,1)$, we see
  by looking at (\ref{eq:21}) that if $\iota_1 P_2$ belongs to
  the orbit of $P_2$, then $\iota_1 P_2 = P_2$. This condition
  is satisfied if and only if $A = \frac{1}{2}$,
  and then $\delta(P_2, \iota_1 P_2) = 0$. Otherwise, $\delta(P_2, \iota_1 P_2) = \bot$.

\item For $A = 0$, since $v(\iota_1 P_2) = (0,-1)$, we see by looking
  at (\ref{eq:22}) that if $\iota_1 P_2$ belongs
  to the orbit of $P_2$, then $\iota_1 P_2 = \sigma^{-2} P_2$. This is always
  the case for any weighting $d_{i,j}$, thus $\delta(P_2, \iota_1 P_2) = -2$.
\end{enumerate}
We thus compute the corresponding entry of Table~\ref{tab:mat_results1} \[
  \delta(P_2, \iota_1 P_2) = \left\{
    \begin{array}{rl}
      0 & \mbox{if $A = \tfrac{1}{2}$} \\
      -2 & \mbox{if $A = 0$} \\
      \bot & \mbox{otherwise.} \end{array}\right.\]
\end{enumerate} 
The other entries are computed in the same way. \exqed
\end{example}

Using this procedure, we manage to compute the entries
of $M_1$ as defined in (\ref{mat:matrices}) for each
set of steps $\calS_i$. These matrices are put in Appendix~\ref{sect:mat_results}.

\section{Classification} \label{sec:classification}

We are now geared to prove the classification.
\subsection{Some decouplings and homogeneous solutions}
\label{sec:some-deco-homog}

We first determine for which supports and weights the
functional equation (\ref{eq:eq_lin_hom}) admits a nonzero
rational solution.
The following two computational lemmas will be used to exhibit particular
signed solutions to (\ref{eq:eq_lin_hom}).
These statements are checked in the Maple worksheet for this section.
Recall the definitions of $\gamtld_1$ and $\gamtld_2$ in~(\ref{eq:20}).
\begin{lemma} \label{lem:dcpl_type12}
  Assume that $d_{1,1} = 0$ (supports $\calS_1$ and $\calS_2$). Write $x = x(s)$
  and $y = y(s)$, and let $\lambda \in \C$. Then the following identities hold:
  \begin{enumerate}[label=(\roman*)]
  \item $u_{\lambda} \eqdef (1-\lambda) - t d_{1,0} x - t d_{1,-1} \tfrac{x}{y} = - (\lambda - t d_{0,1} y - t d_{-1,1} \tfrac{y}{x})$,
  \item $\left( \lambda - A + x \gamtld_1 \right)u_{\lambda} =
    \lambda (1 - \lambda) - t^2 d_{1,-1} d_{-1,1} - (\lambda t d_{1,0} + t^2 d_{1,-1} d_{0,1}) x$,
  \item $- \left(1 - \lambda  -B + y \gamtld_2 \right) u_{\lambda} = \lambda (1 - \lambda) - t^2 d_{1,-1} d_{-1,1} - ((1 - \lambda) t d_{0,1} + t^2 d_{-1,1} d_{1,0}) y$.
  \end{enumerate}
\end{lemma}

\begin{lemma} \label{lem:dcpl_type3}
  Assume that $d_{1,0} = d_{0,1} = 0$ (support $\calS_3$).
  Write $x=x(s)$ and $y=y(s)$.
  Then the following identities hold:
  \begin{enumerate}[label=(\roman*)]
  \item $(\tfrac{1}{2} - A + x \gamtld_1)^2 = \tfrac{1}{4} - d_{1,-1} d_{-1,1} t^2 - d_{1,1} d_{1,-1} t^2 x^2$,
  \item $(\tfrac{1}{2} - B + y \gamtld_2)^2 = \tfrac{1}{4} - d_{1,-1} d_{-1,1} t^2 - d_{1,1} d_{-1,1} t^2 y^2$.
  \end{enumerate}
\end{lemma}

Using these computations
along with Lemma~\ref{lem:wrong_signs_hom},
and the tables of Section~\ref{sect:decide_diff}
along with Lemma~\ref{lem:norel},
we build below Table~\ref{tab:sol_hom}\footnote{$A$ generic means that $A \neq \{0,\tfrac{1}{2}\}$ and $A+B \neq 1$ \label{foot:tbl_generic} and $B$ generic means
  that $B \neq \{0,\tfrac{1}{2}$\}.}
which describes the cases where the homogeneous equation
(\ref{eq:eq_lin_hom}) has solutions.
Its entries are read as follows:
\begin{itemize}
\item Either the entry has the form $(\bot, i)$, meaning
  that row $i$ of $M_1$ is made of $\bot$'s. In this case,
  the application of Lemma~\ref{lem:norel}
  shows that there is no nonzero solution to (\ref{eq:eq_lin_hom}).
\item Either the entry has the form $(\varepsilon_1, \varepsilon_2) \in (\pm,\pm)$,
  which means that there is a nonzero pair
  of fractions $(h_1',h_2')$ satisfying the relation $\gamtld_1(s) h_1'(s) + \gamtld_2(s) h_2'(s) = 0$ with $(h_1')^{\iota_1} / h_1' = \varepsilon_1$
  and $(h_2')^{\iota_2} / h_2' = \varepsilon_2$
  (a \emph{signed solution} to (\ref{eq:eq_lin_hom})).
  If $(\varepsilon_1,\varepsilon_2) \neq (-,-)$, then
  the application of Lemma~\ref{lem:wrong_signs_hom}
  shows that there is no nonzero solution to (\ref{eq:eq_lin_hom}).
\end{itemize}

Hence, Equation~(\ref{eq:eq_lin_hom}) has a nonzero solution if
and only if the entry of Table~\ref{tab:sol_hom}
is $(-,-)$. We thus obtain the following result:
\begin{lemma} \label{lem:clas_hom}
  Equation (\ref{eq:eq_lin_hom}) has a nonzero pair of rational solutions
  if and only if $A=B=\tfrac{1}{2}$ with set of steps $\calS_3$. In this case,
  the solution is given by the pair $(\tfrac{1}{\gamtld_1}, - \tfrac{1}{\gamtld_2})$.
\end{lemma}

\begin{table}[h!]
  \centering
  \begin{tabular}{c|clccl|}
    \cline{2-6}
    \multicolumn{1}{l|}{}                                                   & \multicolumn{1}{c|}{$\calS_1$}      & \multicolumn{1}{c|}{$\calS_2$}         & \multicolumn{1}{c|}{$\calS_3$}      & \multicolumn{1}{c|}{$\calS_4$} & \multicolumn{1}{c|}{$\calS_5$} \\ \hline
    \multicolumn{1}{|c|}{$(A,B)=(0,0)$}                                     & \multicolumn{5}{c|}{$(+,+)$}                                                                                                                      \\ \hline
    \multicolumn{1}{|c|}{$(A,B)=(0,\tfrac{1}{2})$}                          & \multicolumn{1}{c|}{$(-,+)$} & \multicolumn{1}{l|}{$(\bot,4)$} & \multicolumn{1}{c|}{$(+,-)$} & \multicolumn{2}{c|}{$(\bot,4)$}                   \\ \hline
    \multicolumn{1}{|c|}{$(A,B)=(\tfrac{1}{2},0)$}                          & \multicolumn{1}{c|}{$(-,+)$} & \multicolumn{1}{l|}{$(\bot,2)$} & \multicolumn{2}{c|}{$(-,+)$}                           & $(\bot,2)$              \\ \hline
    \multicolumn{1}{|c|}{$(A,B)=(\tfrac{1}{2},\tfrac{1}{2})$}               & \multicolumn{2}{c|}{\multirow{2}{*}{$(+,+)$}}                  & \multicolumn{1}{c|}{$(-,-)$} & \multicolumn{2}{c|}{$(\bot,4)$}                   \\ \cline{1-1} \cline{4-6}
    \multicolumn{1}{|c|}{$A+B=1$ and $(A,B) \neq (\tfrac{1}{2},\tfrac{1}{2})$} & \multicolumn{2}{c|}{}                                          & \multicolumn{3}{c|}{$(\bot,2)$}                                                  \\ \hline
    \multicolumn{1}{|c|}{$A$ generic\footref{foot:tbl_generic}}                                       & \multicolumn{5}{c|}{$(\bot,2)$}                                                                                                                   \\ \hline
    \multicolumn{1}{|c|}{$B$ generic\footref{foot:tbl_generic}}                                       & \multicolumn{5}{c|}{$(\bot,4)$}                                                                                                                   \\ \hline
  \end{tabular}
  \caption{The table summarizing solutions to the homogeneous equation~(\ref{eq:eq_lin_hom}).
    Note that all cases are handled, as $A$ and $B$ are never equal to $1$.
    We refer to the proof below for details on the signed entries.
    \label{tab:sol_hom}}
\end{table}

\begin{proof}[Proof of Table~\ref{tab:sol_hom}]
  First, the entries of type $(\bot, i)$ can be directly checked
  by looking at the tables in computed in Section~\ref{sect:decide_diff}.
  It remains to prove the ``signed'' entries. To do this,
  we will apply Lemmas~\ref{lem:dcpl_type12} and~\ref{lem:dcpl_type3}
  for some sets of steps and various values of $A$ and $B$ to
  exhibit the signed solutions to (\ref{eq:eq_lin_hom}).

  We first tell how to build the first four lines,
  that correspond to the cases $(A,B) \in \{0,\tfrac{1}{2}\}^2$.
  These conditions on $A$ and $B$ correspond to the fact that
  each individual function~$\gamtld_1$ (depending on $A$) and $\gamtld_2$
  (depending on $B$) admits
  a \emph{signed decoupling}. More precisely, assume that we restrict to
  some set of steps $\calS_i$ with $1 \le i \le 5$, and some value
  of $(A,B) \in \{0,\tfrac{1}{2}\}^2$. If one writes
  \[ \gamtld_1 = h_{1,1} \cdot h_{1,2} \mbox{ with $h_{1,1}^{\iota_1} = \varepsilon_{1,1} h_{1,1}$
      and $h_{1,2}^{\iota_2} = \varepsilon_{1,2} h_{1,2}$} \]
  and
  \[ \gamtld_2 = h_{2,1} \cdot h_{2,2} \mbox{ with $h_{2,1}^{\iota_1} = \varepsilon_{2,1} h_{2,1}$
      and $h_{2,2}^{\iota_2} = \varepsilon_{2,2} h_{2,2}$,} \]
  then one obtains the following signed solution to (\ref{eq:eq_lin_hom})
  for models of set of steps $\calS_i$ and $(A,B)$ having the prescribed value: \[
    (h_1,h_2) = \left(\frac{h_{2,1}}{h_{1,1}},\frac{h_{1,2}}{h_{2,2}}\right)
    \mbox{ with $h_1^{\iota_1} = \varepsilon_{1,1} \varepsilon_{2,1} h_1 = \varepsilon_1 h_1$
      and $h_2^{\iota_2} = \varepsilon_{1,2} \varepsilon_{2,2} h_2 = \varepsilon_2 h_2$.}\]
  We thus only give below the signed decouplings of $\gamtld_1$ and $\gamtld_2$
  relatively to the given set of steps and weights $A$, $B$. They cover the
  first four lines of Table~\ref{tab:sol_hom}. We write below $x = x(s)$
  and $y = y(s)$.
  \begin{enumerate}
  \item{\textbf{Signed decouplings of $\boldsymbol{\gamtld_1}$:}}

    \begin{enumerate}

    \item{\textbf{$\boldsymbol{A=0}$, any set of steps:}}
      In this case, we have $\gamtld_1 = - t d_{-1,1} \tfrac{1}{y} \in \C(y)$, hence \[
        \gamtld_1 = h_{1,2} \mbox{ with $h_{1,2}^{\iota_2} = h_{1,2}$.}\]

    \item{\textbf{$\boldsymbol{A=\tfrac{1}{2}}$, sets of steps $\boldsymbol{\calS_1}$, $\boldsymbol{\calS_3}$ and $\boldsymbol{\calS_4}$ (equivalently: $\boldsymbol{d_{1,0} = 0}$)}:}
      In this case, it is easy to check that
      \[
        (x \gamtld_1)^2 = \tfrac{1}{4} - t^2 d_{1,-1} d_{-1,1} - t^2 d_{1,-1} d_{0,1} x - t^2 d_{1,-1} d_{1,1} x^2 \in \C[x].
      \]
      This polynomial is never a square in $\C[x]$:
      \begin{itemize}
      \item When $d_{1,1} = 0$, it has degree $1$ in the variable $x$ because then
        $d_{0,1} \neq 0$.
      \item Otherwise, it has degree $2$ in the variable $x$,
        with discriminant $\Delta$ equal to
        \begin{align*} \Delta
          &= t^4 d_{1,-1}^2 d_{0,1}^2 + (1 - 4 t^2 d_{1,-1} d_{-1,1}) t^2 d_{1,-1} d_{1,1}  \\
          &=  d_{1,-1} d_{1,1} t^2 + (d_{1,-1}^2 d_{0,1}^2 - 4 d_{1,-1} d_{-1,1}) t^4.
        \end{align*}
        The coefficient in $t^2$ is nonzero since $d_{1,-1}$ and $d_{1,1}$ are nonzero.
        Hence, since $t$ is transcendental over the field $\Q(d_{i,j}) \subset \Fpar$,
        we deduce that $\Delta \neq 0$.
      \end{itemize}
      Therefore, $x \gamtld_1$ does not belong to $\C(x(s))$
      while $\left(x \gamtld_1\right)^2$ does. Since $\C(s)/\C(x)$
      is Galois with Galois group $\iota_1$, these conditions translate
      into $\left(x \gamtld_1\right)^{\iota_1} \neq x \gamtld_1$
      and
      $\left(\left(x \gamtld_1\right)^2\right)^{\iota_1} = (x \gamtld_1)^2$,
      so $(x \gamtld_1)^{\iota_1} = - x \gamtld_1$, and
      \[
        \gamtld_1 = h_{1,1} \mbox{ with $h_{1,1}^{\iota_1} = - h_{1,1}$.}
      \]
    \end{enumerate}

  \item{\textbf{Signed decouplings of $\boldsymbol{\gamtld_2}$:}}
    \begin{enumerate}

    \item{\textbf{$\boldsymbol{B=0}$, any set of steps}:}
      In this case, we have $\gamtld_2 = - t d_{-1,1} \tfrac{1}{x} \in \C(x)$,
      and thus \[
        \gamtld_2 = h_{2,1}  \mbox{ with $h_{2,1}^{\iota_1} = h_{2,1}$.}
      \]
    \item{\textbf{$\boldsymbol{B=\tfrac{1}{2}}$, set of steps $\boldsymbol{\calS_1}$}:}
      In this case, from (iii) of Lemma~\ref{lem:dcpl_type12} with $\lambda = \tfrac{1}{2}$
      we have \[
        \mu \eqdef  - (y \gamtld_2) u_{\lambda}
        = - (\tfrac{1}{2} - t d_{-1,1} \tfrac{x}{y}) u_{\lambda} = \tfrac{1}{4} - t^2 d_{1,-1} d_{-1,1}
        - \tfrac{1}{2} t d_{0,1} y \in \C[y].
      \]

      Moreover, from (ii) of Lemma~\ref{lem:dcpl_type12} with $\lambda = \tfrac{1}{2}$,
      then \[
        u_{\lambda}^2 = \tfrac{1}{4} - t^2 d_{1,-1} d_{-1,1} - t^2 d_{1,-1} d_{0,1} x \in \C[x].
      \]
      This polynomial is not a square in $\C(x)$ because it has degree $1$ in $x$,
      so $u_{\lambda}^{\iota_1} = - u_{\lambda}$.
      Reasoning as above, we deduce\[
        \gamtld_2 = h_{2,1} \cdot h_{2,2}
        \eqdef  \tfrac{1}{u_{\lambda}} \cdot \left(-\tfrac{\mu}{y}\right) \mbox{ with $h_{2,1}^{\iota_1} = - h_{2,1}$
          and $h_{2,2}^{\iota_2} = h_{2,2}$.}\]
    \item{\textbf{$\boldsymbol{B=\tfrac{1}{2}}$, set of steps $\boldsymbol{\calS_3}$}:}
      In this case, from (ii) of Lemma~\ref{lem:dcpl_type3} we have
      \[ (y \gamtld_2)^2 = \tfrac{1}{4} - d_{1,-1} d_{-1,1} t^2 - d_{1,1} d_{-1,1} t^2 y^2 \in \C[y].\]
      This polynomial is not a square in $\C[y]$. Indeed,
      it has degree~$2$, and its discriminant $\Delta$ is equal to \[
        \Delta = (1 - 4 d_{1,-1} d_{-1,1} t^2) d_{1,1} d_{-1,1} t^2 = d_{1,1} d_{-1,1} t^2 + O(t^4).\]
      As $t$ is transcendental over the field of parameters $\Q(d_{i,j}) \in \Fpar$,
      then $\Delta$ is always nonzero since $d_{1,1}$ and $d_{-1,1}$ are
      nonzero. Therefore, $(y \gamtld_2)^{\iota_2} = - y \gamtld_2$,
      from which we deduce \[
        \gamtld_2 = h_{2,2} \mbox{ with $h_{2,2}^{\iota_2} = - h_{2,2}$.}\]
    \end{enumerate}
  \end{enumerate}

  There now remains to fill line 5 of Table~\ref{tab:sol_hom},
  which corresponds to the case of $A+B=1$ for sets of steps $\calS_1$ and $\calS_2$.
  In this case, we have from (ii) of Lemma~\ref{lem:dcpl_type12} with $\lambda = A$ that
  \[ (x \gamtld_1) u_{\lambda} = \lambda (1 - \lambda) - t^2 d_{1,-1} d_{-1,1} - (\lambda t d_{1,0} + t^2 d_{1,-1} d_{0,1}) x \in \C[x].\]
  Moreover, from (iii) of Lemma~\ref{lem:dcpl_type12} with $\lambda = A$,
  then
  \[ - (y \gamtld_2) u_{\lambda} = \lambda (1 - \lambda) - t^2 d_{1,-1} d_{-1,1} - ((1-\lambda) t d_{0,1} + t^2 d_{-1,1} d_{1,0}) y \in \C[y].\]
  Note that $u_{\lambda} \neq 0$, for $(x \gamtld_1) u_{\lambda}$ is a nonzero
  polynomial in $\C[x]$ (the constant coefficient
  is a nonzero polynomial in $\Fpar[t]$, for $t$ is transcendental over
  $\Fpar$ and $d_{1,-1} d_{-1,1} \neq 0$),
  $x$ transcendental over $\C$. Therefore, the pair \[
    (h_1,h_2) \eqdef  \left(\frac{x}{x \gamtld_1 u_{\lambda}}, - \frac{y}{y \gamtld_2 u_{\lambda}}\right) \mbox{ with $h_1^{\iota_1} = h_1$ and $h_2^{\iota_2} = h_2$}\]
  is a signed solution to (\ref{eq:eq_lin_hom}).
\end{proof}

\subsection{\texorpdfstring{One particular case: support
    $\boldsymbol{\calS_1}$, $\boldsymbol{B = \frac{1}{2}}$
    and $\boldsymbol{A \neq \frac{1}{2}}$}{One particular case:
    support $S_1$, $B=1/2$, $A \neq 1/2$}}
\label{sec:one-particular-case}

In the previous subsection, we were able to give a uniform
proof for determining which parameters and supports
allow for nonzero solutions to (\ref{eq:eq_lin_hom}).
For (\ref{eq:eq_lin_inhom}), we were not able to find a uniform
argument, for one edge case remains, that we thus treat aside
in this section. The remaining cases (i.e. $\calS_1$ with $B \neq \frac{1}{2}$
or $A = \frac{1}{2}$ or the other set of steps)
are then treated in Section~\ref{sec:full-classification}.

This edge case concerns models with set of steps $\calS_1$
($d_{1,1} = d_{1,0} = 0$)
and Boltzmann weights satisfying
$B = \frac{1}{2}$ and $A \neq \frac{1}{2}$.
We then show that the generating function $Q(x,y)$ is non-D-algebraic
in $x$ or $y$.

Following Lemma~\ref{lem:dcpl_type12}, let
\begin{align*}
  u &\eqdef  u_{1/2} = \tfrac{1}{2} - t d_{1,-1} \tfrac{x}{y} = -( \tfrac{1}{2} - t d_{0,1} y - t d_{-1,1} \tfrac{y}{x}) = \frac{1}{2}-A+x \gamtld_1.
\end{align*}
This function satisfies the following relations.
\begin{lemma} \label{lem:edge_comp}
  Write $x=x(s)$ and $y=y(s)$.
  The following identities hold:
  \begin{enumerate}[label=(\roman*)]
  \item $u^2 = \tfrac{1}{4} - t^2 d_{-1,1} d_{1,-1} - t^2 d_{1,-1} d_{0,1} x \in \C[x]$ and $u^{\iota_1} = - u$,
  \item $- (y \gamtld_2) u = \tfrac{1}{4} - \tfrac{1}{2} t d_{0,1} y - t^2 d_{-1,1} d_{1,-1} \in \C[y]$
    .
  \end{enumerate}
\end{lemma}
\begin{proof}
  The algebraic identities of (i) and (ii) are a direct consequence of
  Lemma~\ref{lem:dcpl_type12}.

  Moreover, $u^2 \in \C(x)$, while
  from (i) this polynomial is not a square in $\C(x)$
  (indeed, it has degree $1$
  in $x$). Since $\C(s)/\C(x)$ is Galois with Galois group generated
  by $\iota_1$, we have $(u^2)^{\iota_1} = u^2$ and $u^{\iota_1} \neq u$,
  hence $u^{\iota_1} = -u$.
\end{proof}

We investigate the solutions of (\ref{eq:eq_lin_inhom}).

\begin{lemma} \label{lem:propag_bound_order}
  If $(h_1,h_2)$ is a solution to (\ref{eq:eq_lin_inhom}),
  then $(h_2)_{\infty} = k \cdot (\iota_2 P_4 + P_4) + p \cdot (0 + \infty)$ for some $p \ge 0$ and $k \in \{0,1\}$.
\end{lemma}
\begin{proof}
  From Table~\ref{tab:mat_results1} and Proposition~\ref{prop:mat_sym},
  we observe that the only nonnegative entry of the
  matrix $M_2$ is $\delta(\iota_2 P_4, P_4) = 1$.
  Therefore, from Lemma~\ref{lem:sandwich_poles},
  we see that if $(h_1, h_2)$ is a pair of solutions to (\ref{eq:eq_lin_inhom}),
  then the poles of $h_2$ must belong to $\{P_4, \iota_2 P_4, 0, \infty\}$.

  We now bound the order of the pole $P_4$ of $h_2$.
  We show that if it has order greater than $1$, then
  $P_4$ is a pole of $h_1$, and deduce a contradiction.
  For the first part, we use that $h_2$ is a solution to (\ref{eq:eq_lin_inhom}):
  \[
    \gamtld_1 h_1 + \gamtld_2 h_2 + \omega = 0.
  \]
  We first note that $P_4$ is a zero of order $1$ of $\gamtld_2$
  (the zeros of $\gamtld_1$ are computed in the Maple worksheet dedicated
  to the set of steps $\calS_1$).
  Now, assume that $P_4$ is a pole of $h_2$ of order greater than $1$.
  As $P_4$ is a zero of order $1$ of $\gamtld_2$,
  we deduce that $P_4$ is a pole of $\gamtld_2 h_2$.
  Then from (\ref{eq:eq_lin_inhom}), we deduce that $P_4$ is a
  pole of $\gamtld_1 h_1$.
  As $P_4 \neq 0, \infty$, it is not a pole of $\gamtld_1$
  (Proposition~\ref{prop:div_gamma}), thus $P_4$ is
  a pole of $h_1$.

  Since $P_4$ is a pole of $h_1$, then by Lemma~\ref{lem:sandwich_poles},
  there must exist $Q^+ \in \cL^+$ such that $\sigma^n P_4 = Q^+$ for some $n \ge 0$.
  But since $B = \frac{1}{2}$, Table~\ref{tab:mat_results1} implies
  that $\delta(\iota_2 P_4, P_4) = \delta(\sigma^{-1} P_4, P_4) = 1$.
  Therefore, we deduce that
  \[
    \delta(\iota_2 P_4, Q^+) = \delta(\iota_2 P_4, P_4) + \delta(P_4, Q^+) = 1+n \ge 1.
  \]
  This is a contradiction, since no entry of the
  line corresponding to $\iota_2 P_4$ in Table~\ref{tab:mat_results1}
  is positive.

  Therefore, the pole $P_4$ has order $0$ or $1$. Since $h_2$
  satisfies $h_2^{\iota_2} = h_2$, the point $\iota_2 P_4$
  has the same order as $P_4$ as a pole of $h_2$, and Table~\ref{tab:mat_results1}
  and the fact that $B = \frac{1}{2}$ asserts that $P_4 \neq \iota_2 P_4$,
  hence the result.
\end{proof}

\begin{lemma} \label{lem:sup1_div}
  The function $u \gamtld_2$ has divisor
  $(u \gamtld_2) = P_4 + \iota_2 P_4 - 0 - \infty$.
\end{lemma}
\begin{proof}
  See the dedicated section in the Maple worksheet covering the set
  of steps $\calS_1$.
\end{proof}

We can now state the classification for this support and parameters.
\begin{proposition} \label{prop:clas_b12}
  For every weighted model of the set of steps $\calS_1$,
  if $B=\tfrac{1}{2}$ and $A \neq \tfrac{1}{2}$, then
  the series $Q(x,y)$ is non-D-algebraic in $x$ and $y$.
\end{proposition}
\begin{proof}
  Assume that $Q(x,y)$ is $x$-D-algebraic or $y$-D-algebraic.
  By Proposition~\ref{prop:Qx_Fs}, Theorem~\ref{thm:ishizaki}
  and Lemma~\ref{lem:rat_sol_dcpl}, then
  $\Ftld(s) = \ftld(s) + \ftld_h(s)$, $\Gtld(s) = \gtld(s) + \gtld_h(s)$,
  with $(\ftld,\gtld)$ solution to (\ref{eq:eq_lin_inhom})
  and $(\ftld_h,\gtld_h)$ solution to (\ref{eq:eq_lin_hom}).
  As (\ref{eq:eq_lin_hom}) has no nonzero rational solution
  by Lemma~\ref{lem:clas_hom},
  this implies that $\ftld_h = \gtld_h = 0$, thus
  $\Ftld(s) = \ftld(s)$ and $\Gtld(s) = \gtld(s)$.
  We now distinguish between two cases, depending on the
  value of $k$ in Lemma~\ref{lem:propag_bound_order} ($k=0$ or $k=1$).
  \begin{itemize}
  \item If $k=0$, then $(\gtld)_{\infty} = p \cdot (0 + \infty)$.
    Thus, by Lemma~\ref{lem:lift_laurent}, we have that
    $\gtld(s) = H(1/y(s))$ for $H(y) \in \Fpar[y]$ a polynomial.
    Thus, Proposition~\ref{prop:lift_an_pow} implies that
    $y Q(0,y) = H(1/y)$,
    which is absurd since $Q(0,y) = 1 + O(y)$.

  \item Otherwise, $k=1$, and
    $(\gtld)_{\infty} = P_4 + \iota_2 P_4 + p \cdot (0 + \infty)$.
    We thus have from Lemma~\ref{lem:sup1_div}
    \[(\gtld \cdot (u \gamtld_2)) = (\gtld)_0 + P_4 + \iota_2 P_4 - P_4 - \iota_2 P_4 - (p+1) \cdot (0 + \infty) = (\gtld)_0 - (p+1) \cdot (0 + \infty).\]
    Hence, the poles of $\gtld \cdot (u \gamtld_2)$ belong to $\{0, \infty\}$.
    Furthermore, we have from (ii) of Lemma~\ref{lem:edge_comp} that
    $- 4 u \gamtld_2 = \frac{1 - 2 t d_{0,1} y(s) - 4 t^2 d_{-1,1} d_{1,-1}}{y(s)}
    \in \C(y(s))$,
    thus $\gtld \cdot (u \gamtld_2)$ is fixed by $\iota_2$. Thus,
    Lemma~\ref{lem:lift_laurent} implies that
    $\gtld \cdot (- 4 u \gamtld_2) = H(1 / y(s))$
    for some polynomial $H(y) \in \Fpar[y]$, and thus
    \[
      \gtld(s) = \frac{y(s)}{1 - 2 t d_{0,1} y(s) -
        4 t^2 d_{-1,1} d_{1,-1}} H(1/y(s))
    \]
    for some polynomial $H(y) \in \Fpar[y]$. Thus,
    Proposition~\ref{prop:lift_an_pow} implies that
    \[
      y Q(0,y) = \frac{y}{1 - 2t d_{0,1} y - 4t^2 d_{-1,1} d_{1,-1}} H(1/y).
    \]
    Since
    \[
      y Q(0,y) = y + O(y^2)
    \]
    and
    \[
      \frac{y}{1-2t d_{0,1} y - 4 t^2 d_{-1,1} d_{1,-1}} H(1/y) = \frac{H(0)}{1-4t^2 d_{-1,1} d_{1,-1}} + O(1/y),
    \]
    this implies that $H(y) = \mu$ a constant in $\Fpar$,
    and thus that $\gtld(s) = \mu / u$ for some $\mu \in \Fpar$.

    Therefore, we may rewrite (\ref{eq:eq_lin_inhom}) as
    \begin{equation} \label{eq:edge_hypertrans_1}
      \gamtld_1 \ftld - \tfrac{\mu}{u} + \omega = 0.
    \end{equation}
    But $u^{\iota_1} = -u$ by (i) of Lemma~\ref{lem:edge_comp},
    thus $(x \gamtld_1)^{\iota_1} = (A - \tfrac{1}{2} + u)^{\iota_1}
    = A - \tfrac{1}{2} - u$, thus $(\gamtld_1)^{\iota_1} + (\gamtld_1) = (2A-1)/x$.
    Moreover, $\omega = 1 - A - B = \frac{1}{2} - A$ since $B = \frac{1}{2}$.
    Thus, by taking $\iota_1 (\ref{eq:edge_hypertrans_1}) + (\ref{eq:edge_hypertrans_1})$,
    one obtains the identity
    \[
      (2A-1)\, \tfrac{\ftld}{x} - (2A-1) = 0.
    \]
    As $A \neq \tfrac{1}{2}$, this implies that
    $\ftld(s) = x(s)$. By Proposition~\ref{prop:lift_an_pow},
    this implies that $Q(x,0) = 1$, which is absurd
    since $a, b > 0$.\qedhere
  \end{itemize}
\end{proof}

\subsection{Full classification}\label{sec:full-classification}

We now state and prove the full classification:
\begin{theorem} \label{thm:thm_clas}
  For any weighted genus $0$ model, the generating function
  $Q(x,y)$ of weighted walks in the quadrant with interacting
  boundaries has the following nature in the variables $x$ and $y$:

  \begin{enumerate}
  \item\label{thm:thm_clas:first} For all models of set of steps $\calS_1$ or $\calS_2$, and if $a+b=ab$, the
    generating function $Q(x,y)$ is \textbf{rational}
    with partial series $Q(x,0)$ and $Q(0,y)$ respectively equal to
    \begin{align*}
      Q(x,0) &=\frac{1}{1 - \displaystyle x \frac{a d_{1,0} t + ab d_{1,-1} d_{0,1} t^2}{1 - ab d_{1,-1} d_{-1,1} t^2}},
      &Q(0,y) &= \frac{1}{1 - \displaystyle y \frac{b d_{0,1} t + ab d_{-1,1} d_{1,0} t^2}{1 - ab d_{1,-1} d_{-1,1} t^2}}.
    \end{align*}
  \item\label{thm:thm_clas:second} For all models of set of steps  $\calS_3$ where $a=b=2$, the generating function $Q(x,y)$ is
    \textbf{algebraic} of degree $4$,
    with partial series $Q(x,0)$ and $Q(0,y)$ respectively equal to
    \begin{align*}
      Q(x,0) &= \frac{1}{\sqrt{1 - \displaystyle x^2 \frac{4 d_{1,1} d_{1,-1} t^2}{1-4 d_{1,-1} d_{-1,1} t^2}}},
      & Q(0,y) &= \frac{1}{\sqrt{1 - \displaystyle y^2 \frac{4 d_{1,1} d_{-1,1} t^2}{1-4 d_{1,-1} d_{-1,1} t^2}}}.
    \end{align*}
  \item\label{thm:thm_clas:third} In all other cases, the series $Q(x,y)$ is \textbf{neither $\boldsymbol{x}$-D-algebraic nor
      $\boldsymbol{y}$-D-algebraic}.
  \end{enumerate}
\end{theorem}
\begin{proof}
  We prove all the points in order. We begin with \ref{thm:thm_clas:first}.
  Recall that as $A = 1 - \frac{1}{a}$ and $B = 1-\frac{1}{b}$,
  we have $a+b=ab$ is equivalent to $A+B = 1$, which implies
  that $\omega = 0$.
  From (ii) and (iii) of Lemma~\ref{lem:dcpl_type12} with $\lambda = A$,
  we see that the pair
  $\left(\frac{x(s)}{x(s) \gamtld_1 u_{A}}, - \frac{y(s)}{y(s) \gamtld_2 u_{A}}\right)~\in~\C(x(s))~\times~\C(y(s))$ is a
  nonzero solution to (\ref{eq:eq_lin_inhom}).
  Therefore, by (1) of Lemma~\ref{lem:lift_sol}
  and (i) and (ii) of Lemma~\ref{lem:dcpl_type12},
  there exists $\lambda \in \C$ such that
  \begin{align*}
    Q(x,0) &= \frac{\lambda}{AB - t^2 d_{1,-1} d_{0,1} x - t^2 d_{1,-1} d_{-1,1} - A t d_{1,0} x} \\
    \mbox{and } Q(0,y) &= \frac{\lambda}{AB - t^2 d_{1,-1} d_{1,0} y - t^2 d_{1,-1} d_{-1,1} - B t d_{0,1} y}.
  \end{align*}
  We know that $Q(0,0)=1$, hence by substituting $x=0$ in $Q(x,0)$ (or $y=0$ in $Q(0,y)$),
  we find $\lambda = AB - t^2 d_{1,-1} d_{-1,1}$. We obtain the identities
  claimed in the theorem using $\tfrac{1}{A} = b$ and $\tfrac{1}{B} = a$
  (this uses $a+b=ab$).

  We now prove \ref{thm:thm_clas:second}.
  In this case, we have from Lemma~\ref{lem:clas_hom} that
  (\ref{eq:eq_lin_hom}) admits
  a nonzero solution $\left(\frac{1}{\gamtld_1}, - \frac{1}{\gamtld_2}\right)$.
  Therefore, by (2) of Lemma~\ref{lem:lift_sol}
  and Lemma~\ref{lem:dcpl_type3},
  there exists $\lambda \in \C$ such that
  \begin{align*}
    Q(x,0) &= \frac{\lambda}{\sqrt{\tfrac{1}{4} - d_{1,-1} d_{-1,1} t^2 - d_{1,1} d_{-1,1} x^2 t^2}}&
                                                                                                      \mbox{and } & Q(0,y) &= \frac{\lambda}{\sqrt{\tfrac{1}{4} - d_{1,-1} d_{-1,1} t^2 - d_{1,1} d_{-1,1} y^2 t^2}}.
  \end{align*}
  We know that $Q(0,0) = 1$, thus $\lambda = \sqrt{\frac{1}{4} - d_{1,-1} d_{-1,1} t^2}$,
  and we get the expression in the theorem.

  Now, it remains to prove \ref{thm:thm_clas:third}, namely that for all other cases $Q(x,y)$ is
  non-D-algebraic in $x$ and $y$. Depending on the case,
  we use one of the three arguments below.
  \begin{enumerate}
  \item[(1)] If we are in the situation of Section~\ref{sec:one-particular-case},
    then $Q(x,y)$ is non-D-algebraic in $x$ and $y$ by Proposition~\ref{prop:clas_b12}.
  \item[(2)] If the entries of $M_1$ satisfy $(M_1)_{i,j} \in \Z^- \cup \{\bot\}$,
    then by Lemma~\ref{lem:noseg_nopol},
    if $(h_1,h_2)$ is a rational solution to (\ref{eq:eq_lin_inhom}),
    the poles of $h_1$ must belong to $\{0, \infty\}$. Moreover, Lemma~\ref{lem:clas_hom}
    asserts that there is no nonzero rational solution to (\ref{eq:eq_lin_hom})
    (the only situation where it happens corresponds to
point~\ref{thm:thm_clas:second}
    of the present theorem, and has already been covered).
    Therefore, by Lemma~\ref{lem:hypertrans}, the generating function
    $Q(x,y)$ is non-D-algebraic in $x$ and $y$.
  \item[(2')] Likewise, if the entries of $M_2$ satisfy $(M_2)_{i,j} \in \Z^- \cup \{\bot\}$,
    then for $(h_1,h_2)$ a solution to (\ref{eq:eq_lin_inhom}) the poles of
    $h_2$ must belong to $\{0, \infty\}$ by Lemma~\ref{lem:noseg_nopol}.
    Moreover, Lemma~\ref{lem:clas_hom} asserts that there is no nonzero
    rational solution to (\ref{eq:eq_lin_hom}).
    Therefore, by Lemma~\ref{lem:hypertrans},
    the generating function $Q(x,y)$ is non-D-algebraic in $x$ and $y$.
  \end{enumerate}
  We check that for any value of the parameters, we are in one of the
  three above cases (see Section~\ref{sect:decide_diff}).

  \begin{center}
    \mbox{}\clap{
      \begin{tabular}{c|cl|c|cc|}
        \cline{2-6}
        \multicolumn{1}{l|}{} & \multicolumn{1}{c|}{Support 1} & \multicolumn{1}{c|}{Support 2}
        & \multicolumn{1}{c|}{Support 3} & \multicolumn{1}{c|}{Support 4} & \multicolumn{1}{c|}{Support 5} \\ \hline
        \multicolumn{1}{|l|}{(1) Proposition~\ref{prop:clas_b12}} & \multicolumn{1}{c|}{$A+B \neq 1 \land B = \tfrac{1}{2}$}
                                                               & \multicolumn{1}{c|}{} & & \multicolumn{2}{c|}{} \\ \hline
        \multicolumn{1}{|l|}{(2) $(M_1)_{i,j} \in \Z^- \cup \{\bot\}$}         & \multicolumn{1}{c|}{}                                       & $A+B \neq 1$            & \multicolumn{1}{c|}{$A \neq \tfrac{1}{2}$} & \multicolumn{2}{l|}{}        \\ \hline
        \multicolumn{1}{|l|}{(2') $(M_2)_{i,j} \in \Z^- \cup \{\bot\}$}         & \multicolumn{1}{c|}{$A+B \neq 1 \land B \neq \tfrac{1}{2}$} &                         & $B \neq \tfrac{1}{2}$                      & \multicolumn{2}{c|}{always}  \\ \hline \hline
        \multicolumn{1}{|c|}{Algebraic solution} & \multicolumn{2}{c|}{$A+B=1$}                                                          & $A=B=\tfrac{1}{2}$                         & \multicolumn{2}{l|}{}        \\ \hline
      \end{tabular}
    }
  \end{center}
  The above table gives the exact argument for each value of the
  parameters, and one can check that no case is missing.
\end{proof}

\section{Conclusion and comments} \label{sec:conclusion-comments}

In Theorem~\ref{thm:thm_clas}, we showed how the addition of the Boltzmann
weights affects the nature of the generating function $Q(x,y)$ of
walks with interacting boundaries for weighted models of genus~$0$. Namely,
for the first two sets of steps, the relation $a+b=ab$ between the weights
makes the series $Q(x,y)$ rational; for the third set of steps the relation
$a=b=2$ makes the series $Q(x,y)$ algebraic; while other Boltzmann weights
and other sets of steps keep the series non $x$-D-finite
nor $y$-D-finite. We now give some perspectives based on these results.

\subsection{Phase transitions} \label{sec:phase-transitions}
Regarding the sets of steps $\calS_1$ and $\calS_2$, one may note that since
there is an infinite number of Boltzmann weights $a$, $b$ such that
$Q(x,y)$ is explicit, the question of phase transitions introduced
in \cite{tabbaraExactSolutionTwo2014} can be partially treated on
the curve $a+b=ab$ (a hyperbola).

Recall that the phases are defined as follows.
Let $\calS$ be a weighted model.
For $n \ge 0$, denote $\P_n$ the probability
on the walks using $n$ steps defined by
\[\P_n(w) = \frac{\left(\prod_{(i,j) \in \calS} d_{i,j}^{n_{i,j}} \right) a^{n_x} b^{n_y}}{[t^n] Q(1,1)} \]
(i.e. the probability of a walk using $n$ steps is proportional
to the numerator in the above equation).
Define\[
  \calA \eqdef  \limsup_n\, \P_n(\{\mbox{$w$ walk of $n$ steps} \st \mbox{$w$ terminates on the $x$-axis}\})
\]
and \[
  \calB \eqdef  \limsup_n\, \P_n(\{\mbox{$w$ walk of $n$ steps} \st \mbox{$w$ terminates
    on the $y$-axis}\}).
\]
These limits correspond respectively to $\calA$ and $\calC$
in~\cite{tabbaraExactSolutionTwo2014}).
Four phases are then defined as follows:
\begin{enumerate}
\item if $\calA = \calB = 0$, then the phase is \emph{free}
  (the walk moves away from the axes),
\item if $\calA > 0$ and $\calB = 0$, then the phase is \emph{$x$-attracted}
  (the walk moves away from the $y$-axis,
  and tends to come back infinitely often on the $x$-axis),
\item if $\calA = 0$ and $\calB > 0$, then the phase is \emph{$y$-attracted}
  (the walk moves away from the $x$-axis, and tends to come back infinitely often on the $y$-axis),
\item if $\calA > 0$ and $\calB > 0$, then the phase is \emph{supercritical}
  (the walk tends to come in contact with both axes infinitely often).
\end{enumerate}
The values $\calA$ and $\calB$ can be expressed using the generating
function $Q(x,y)$ as
\begin{align*}
  \calA = \limsup_n\,\frac{[t^n] Q(1,0)}{[t^n] Q(1,1)}
  & & \calB = \limsup_n\,\frac{[t^n] Q(0,1)}{[t^n] Q(1,1)}.
\end{align*}

\begin{figure}[h]
  \begin{subfigure}[t]{.5\textwidth}
    \centering
    \includegraphics[height=4cm]{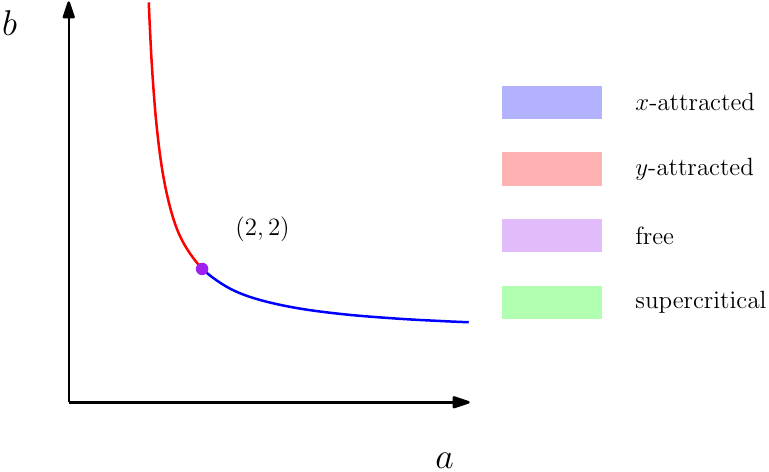}
    \caption{Phase diagram on ${a+b=ab}$.} \label{fig:phase_diagram}
  \end{subfigure}
  \hfill
  \begin{subfigure}[t]{.4\textwidth}
    \centering
    \includegraphics[height=4cm]{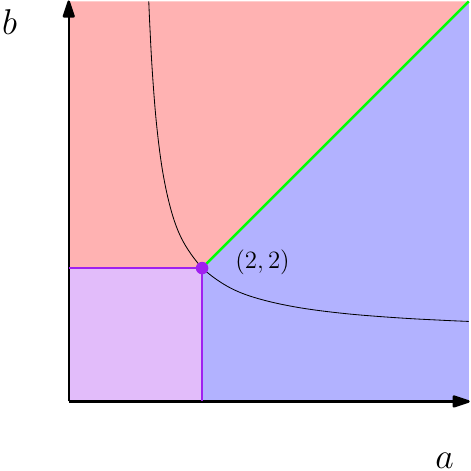}
    \caption{Conjectured phase diagram.} \label{fig:phase_diagram_conj}
  \end{subfigure}
  \caption{Phase diagrams for model $\calS_2$ with $d_{-1,1} = d_{1,-1} = d_{1,0} = d_{0,1} = 1$.}
\end{figure}

In the case of $\calS_1$ or $\calS_2$, the function
$Q(x,y)$ is rational, hence the singularity analysis on the poles
is straightforward.
For instance, for the model $\calS_2$ with $d_{i,j} = 1$,
it yields the phase diagram in Figure~\ref{fig:phase_diagram} below
(see the Maple worksheet).

The change of nature of the generating function $Q(x,y)$ on this curve,
which contains the critical point $(a_0,b_0)$ at the junction of the
four phases suggests that this curve could be related to the phase transitions
of the walk. Numerical computations based on the first
coefficients of $Q(x,y)$ allow us to conjecture that the
full phase diagram looks like Figure~\ref{fig:phase_diagram_conj}.

It would be interesting to know whether the knowledge
of the phases on the curve $a+b=ab$
is enough to deduce some parts of the phase diagram \ref{fig:phase_diagram_conj},
mainly the part under the curve.

\subsection{Combinatorial interpretation}

For models $\calS_1$, $\calS_2$ and $\calS_3$,
we found $\N$-algebraic solutions for $Q(x,y)$
when the weights are subject to some relations
(i.e. $a+b=ab$ for $\calS_1$
and $\calS_2$; $a=b=2$ for $\calS_3$).
These relations were found indirectly through the
study of the $q$-difference equation.
These relations being simple enough, one may wonder
if the expressions found in (\ref{thm:thm_clas:first}) and (\ref{thm:thm_clas:second})
of Theorem~\ref{thm:thm_clas}
for those weights may be deduced through a more combinatorial
argument.

Regarding the weights $a=b=2$ for $\calS_3$,
Andrew Elvey-Price pointed out in a private communication with the
author a direct proof through an adaptation of the reflection principle.
The generating function of unconstrained two dimensional walks
using the set of steps $\calS_3$ that terminate on the $x$-axis
is easily found to be
\[
  F(x, t) = \frac{1}{\displaystyle\sqrt{ 1 - 4 x t d_{1,-1} \left( x t d_{1,1} + \tfrac{1}{x} t d_{-1,1}\right)}}.
\]
A clever adaptation of the reflection principle allows us to relate walks with interacting
boundaries with Boltzmann weights $a=b=2$ to these walks, through
the following identity
\[
  Q(x,0) \cdot \frac{1}{\displaystyle \sqrt{1 - 4 t^2 d_{1,-1} d_{-1,1}}} = F(x,t),
\]
which allows us to deduce the form of $Q(x,0)$. We detail this argument and
extend it to other sets of steps in an upcoming paper.

Such a direct proof is yet to be found regarding the weights $a+b=ab$,
and would be enlightening. For instance, it could give an alternative
explanation as why the relation $a+b=ab$ changes the nature of
$Q(x,y)$, and hopefully permit to find other
sets of steps for which this relation between the Boltzmann weights
yield a D-algebraic generating function.

\subsection{Other \texorpdfstring{$\boldsymbol{q}$-difference}{q-difference} equations} \label{sec:other-q-difference}

The general study of rational solutions to $q$-difference equations
has already been investigated before. Depending on different
constraints on the coefficients and the relation of the complex
number $q$ with regards to these coefficients, there may or may
not be a general algorithm to decide whether such solutions exist.

In the general case where the coefficients of the equation
and $q$ may share algebraic relations, the problem is undecidable
\cite{abramovDecidableUndecidableProblems2010}.

In a more specific case, when the coefficients depend on one parameter,
\cite{abramovLinearQdifferenceEquations2013} gives an algorithm to determine
numerical values of this parameter so that the $q$-difference
equation has a nontrivial
rational solution.
Since we work with more parameters, and we want to find all the algebraic
relations between them so that the equation has solutions, none of these
algorithms can be applied verbatim. This justifies the approach taken
in this paper.

The author thinks that the approach taken in Section~\ref{sect:decoupl_eq},
specifically the structure given by Lemma~\ref{lem:sandwich_poles},
might adapt quite easily to the study of other
decoupling equations of mixed type (multiplicative and additive).
Moreover, we note that our approach works for a general
infinite group of the walk, even if $\iota_2 \iota_1$ is not presented
as a multiplication by $q$ on $\P^1$.
More precisely, for a general decoupling equation of the form
\[
  u h_1 + v h_2 + w = 0
\]
for functions $u$, $v$, $w$, $h_1^{\iota_1} = \pm h_1$
and $h_2^{\iota_2} = \pm h_2$ on some curve $\calC \subset \P^1 \times \P^1$,
the same technique yields
similar finite sets $\cL^-_1$, $\cL^+_1$,
$\cL^-_2$ and $\cL^+_2$, which may be exploited
in the same way as in Section~\ref{sect:decoupl_eq}.

\subsection{Extension to models of genus $1$}
For the time being, we only performed the systematic classification
for walks with interacting boundaries of small steps of genus zero.
It would seem natural to extend the methods for the models with
small steps of genus~$1$ of \cite{HS,dreyfus2019differential,KurkRasch}
for the same purpose.

\subsection*{Acknowledgment} I warmly thank my two advisors Mireille
Bousquet-Mélou and Charlotte Hardouin for their support and
proof reading.

\appendix

\section{Data} \label{sect:mat_results}
This appendix shows the result of the computation of the matrix $M_1$ as
defined in (\ref{mat:matrices}) for each set of step $\calS_i$.
For clarity, the genus $0$ models are subdivided
according to their support, and there
are thus five different tables.
For each table, the value of every entry depends algebraically on the complex
parameters $A=1-\tfrac{1}{a}$,
$B=1-\tfrac{1}{b}$
and $d_{i,j}$. They are computed for each model in its dedicated Maple worksheet.

Note that from Proposition~\ref{prop:mat_sym},
the matrix $M_1$ is symmetric so only the entries
on the upper diagonal are specified, and a simple computation
allows us to deduce the entries of $M_2$ from those
of $M_1$.
Note that for each table, the zeros of $\gamtld_1$ that are $P_1$, $P_2$,
and the zeros of $\gamtld_2$ that are $P_3$, $P_4$, are chosen arbitrarily
and fixed once and for all for each model.

\begin{table}
\centering
\includegraphics[height=9cm]{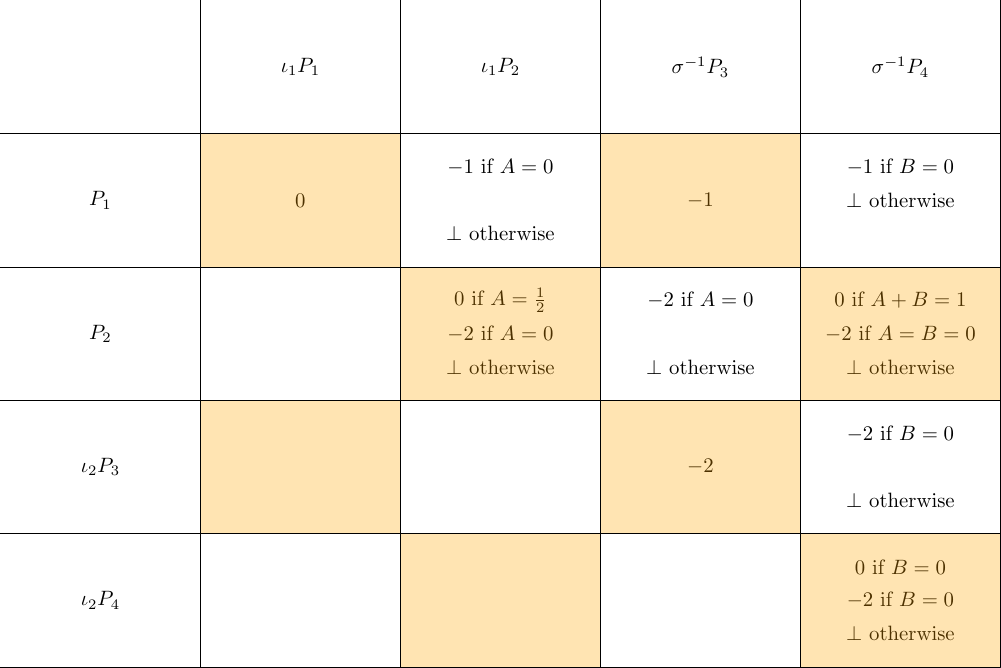}
\caption{Set of steps $\calS_1$} \label{tab:mat_results1}
\end{table}

\begin{table}
\centering
\includegraphics[height=9cm]{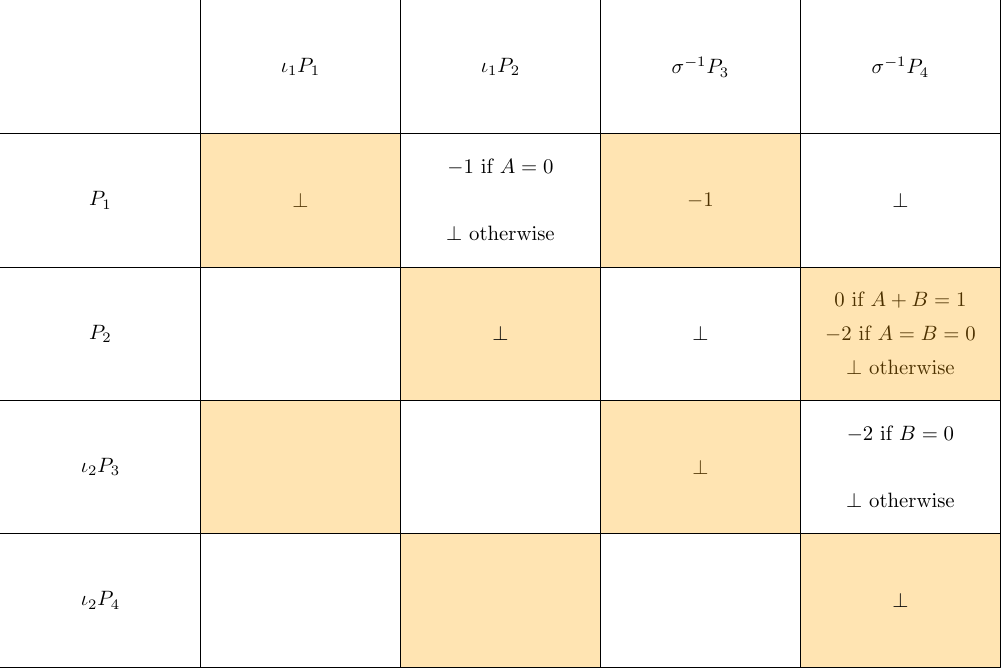}
\caption{Set of steps $\calS_2$} \label{tab:mat_results2}
\end{table}

\begin{table}
\centering
\includegraphics[height=9cm]{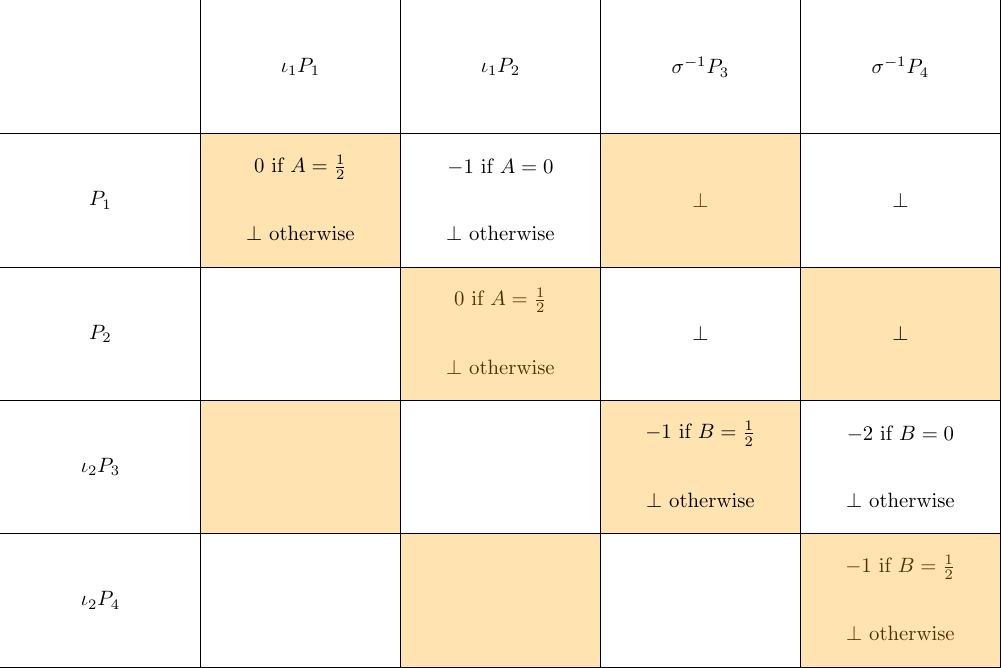}
\caption{Set of steps $\calS_3$} \label{tab:mat_results3}
\end{table}

\begin{table}
\centering
\includegraphics[height=9cm]{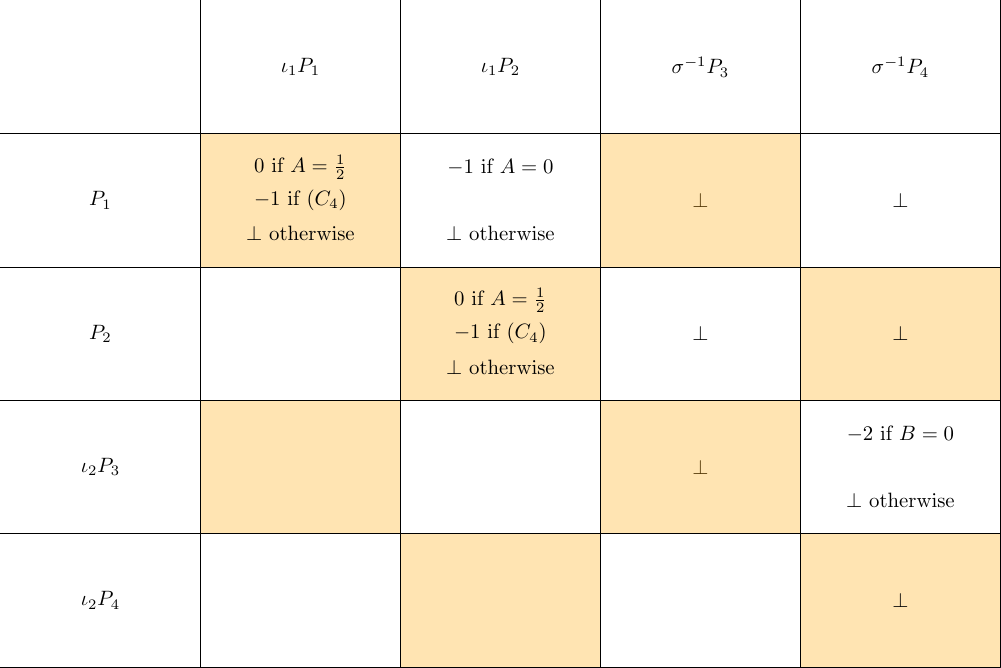}

where $(C_4) \equiv (A=0) \land (4 d_{1,-1} d_{1,1} = {d_{0,1}}^2)$.
\caption{Set of steps $\calS_4$} \label{tab:mat_results4}
\end{table}

\begin{table}
\centering
\includegraphics[height=9cm]{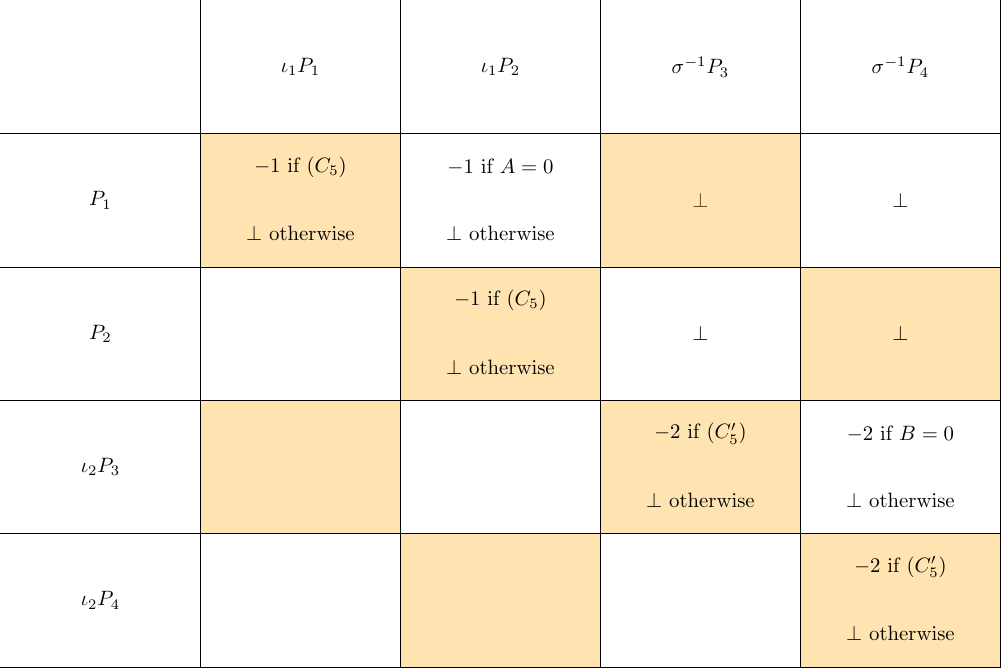}

where $(C_5) \equiv (A=0) \land (4 d_{1,1} d_{-1,1} = {d_{0,1}}^2)$ and $(C_5') \equiv (B=0) \land (4 d_{1,1} d_{-1,1} = {d_{1,0}}^2)$.
\caption{Set of steps $\calS_5$} \label{tab:mat_results5}
\end{table}

\printbibliography

\end{document}